\documentclass[10pt]{amsart}
\usepackage{latexsym}
\usepackage{amscd,amssymb,amsmath,amsxtra}
\usepackage[margin=3cm]{geometry}
\usepackage{graphicx}
\usepackage[curve]{xypic}
\usepackage{stmaryrd}

\makeatletter
\let\@wraptoccontribs\wraptoccontribs
\makeatother
\usepackage{amsmath}
\usepackage{epsfig}
\usepackage{graphicx}
\usepackage{eufrak}
\usepackage{dsfont}
\usepackage[english]{babel}
\usepackage{color}
\usepackage{mathabx}

\usepackage{palatino, mathpazo}
\usepackage{amscd}      % to be able to use "CD" environment
\usepackage{amssymb}
\usepackage{dsfont}
\usepackage{xypic}      % to be able to use "diagram" environment of
\LaTeXdiagrams          % xypic to typeset commutative diagrams
\usepackage[all,v2]{xy}
\xyoption{2cell} \UseAllTwocells \xyoption{frame} \CompileMatrices
\usepackage{latexsym}
\usepackage{epsfig}
\usepackage{enumerate}

\usepackage{latexsym}
\usepackage{epsfig}
\usepackage{amsfonts}
\usepackage{enumerate}
\usepackage{mathrsfs}

\allowdisplaybreaks[3]

\newtheorem{prop}{Proposition}[section]
\newtheorem{lem}[prop]{Lemma}

\newtheorem{thm}[prop]{Theorem}
\newtheorem{rmk}[prop]{Remark}

\newtheorem{defn}[prop]{Definition}

\newtheorem{con}[prop]{Conjecture}

            {\nolinebreak $\Box$ \end{trivlist}}

\newcommand{\noprint}[1]{}

\newcommand{\Hom}{\mbox{Hom}}

\newcommand{\tw}{\mbox{\tiny tw}}

\newcommand{\Z}{\mathop{\sf Z}\nolimits}
\newcommand{\N}{\mathcal{N}}

\newcommand{\XX}{{\mathfrak X}}

\renewcommand{\SS}{{\mathfrak S}}

\newcommand{\YY}{{\mathfrak Y}}

\newcommand{\zz}{{\mathbb Z}}
\newcommand{\hh}{{\mathbb H}}

\newcommand{\nn}{{\mathbb N}}

\newcommand{\qq}{{\mathbb Q}}
\newcommand{\pp}{{\mathbb P}}
\newcommand{\cc}{{\mathbb C}}

\newcommand{\Gm}{{{\mathbb G}_{\mbox{\tiny\rm m}}}}

\newcommand{\sE}{{\mathcal E}}

\newcommand{\sL}{{\mathcal L}}

\newcommand{\sP}{{\mathcal P}}

\newcommand{\sO}{{\mathcal O}}

\newcommand{\sM}{{\mathcal M}}

\newcommand{\sF}{{\mathcal F}}
\newcommand{\sA}{{\mathcal A}}

\newcommand{\Coh}{\mbox{Coh}}

\newcommand{\sB}{\mathscr{B}}

\newcommand{\cHom}{\mathscr{H}om}

\DeclareMathOperator{\Hilb}{Hilb}

\DeclareMathOperator{\ind}{ind}

\DeclareMathOperator{\VW}{VW}
\DeclareMathOperator{\vw}{vw}

\DeclareMathOperator{\vir}{vir}

\DeclareMathOperator{\Pic}{Pic}

\DeclareMathOperator{\Higg}{Higg}

\DeclareMathOperator{\Tot}{Tot}
\DeclareMathOperator{\Br}{Br}

\DeclareMathOperator{\JS}{JS}

\DeclareMathOperator{\NS}{NS}

\DeclareMathOperator{\ess}{ess}
\DeclareMathOperator{\opt}{opt}

\DeclareMathOperator{\SU}{SU}
\DeclareMathOperator{\SL}{SL}
\DeclareMathOperator{\PGL}{PGL}

\newcommand{\rk}{\mathop{\rm rk}}

\newcommand{\tr}{\mathop{\rm tr}\nolimits}

\newcommand{\Jac}{\mathop{\rm Jac}\nolimits}

\newcommand{\tor}{\mathop{\rm tor}\nolimits}

\numberwithin{equation}{subsection}

\newcommand {\mat}      [1] {\left(\begin{array}{#1}}
\newcommand {\rix}          {\end{array}\right)}

\title[S-duality conjecture for K3 surfaces]{A proof of all ranks S-duality conjecture for K3 surfaces}
\author{Yunfeng Jiang}
\address{Department of Mathematics\\ University of Kansas\\ 405 Snow Hall 1460 Jayhawk Blvd\\Lawrence KS 66045 USA} 
\email{y.jiang@ku.edu}
\author{Hsian-Hua Tseng}
\address{Department of Mathematics\\ Ohio State University\\ 100 Math Tower 231 West 18th
Ave.\\ Columbus OH 43210 USA}
\email{hhtseng@math.ohio-state.edu}

\begin{document}
\sloppy \maketitle
\begin{abstract}
Using  the multiple cover formula of Y. Toda  for  counting invariants of semistable  twisted sheaves over twisted local K3 surfaces we 
calculate the  $\SU(r)/\zz_r$-Vafa-Witten  invariants for K3 surfaces for any rank $r$   for the Langlands dual group 
$\SU(r)/\zz_r$ of the gauge group $\SU(r)$.  We generalize and prove the S-duality conjecture of Vafa-Witten for K3 surfaces in any rank $r$ based on the result of Tanaka-Thomas for the  $\SU(r)$-Vafa-Witten invariants.
\end{abstract}

%%% -----------------------------------------------------------------------
\maketitle
%%% ----------------------------------------------------------------------

\tableofcontents

\section{Introduction}

The main goal of  this paper is to prove the S-duality conjecture of Vafa-Witten in any  rank $r$ for K3 surfaces.  We calculate the  $\SU(r)/\zz_r$-Vafa-Witten invariants defined in \cite{Jiang_2019} and prove the S-duality conjecture of Vafa-Witten in \cite{VW} comparing with the $S$-transformation formula of 
Tanaka-Thomas's partition function for the $\SU(r)$-Vafa-Witten invariants in \cite{TT2}.  We use  the twisted version of Toda's multiple cover formula for the counting invariants of semistable twisted sheaves on a local twisted K3 surface, which is proved in a separate paper \cite{Jiang-Tseng_multiple}.

We first  review the S-duality conjecture of $N =4$ supersymmetric Yang-Mills theory on a real 4-manifold $M$  in \cite{VW}. 
The theory involves coupling constants $\theta, g$ combined as 
$$\tau:= \frac{\theta}{2\pi} + \frac{4\pi i}{g^2}.$$
The S-duality conjecture  predicts that the transformation $\tau\to -\frac{1}{\tau}$ maps the partition function with gauge group $G$ to the partition function with Langlands dual gauge group $^{L}G$.  Vafa-Witten \cite{VW} considered a $4$-manifold $M$ underlying a smooth projective surface $S$ over $\cc$ and the gauge group $G =\SU(r)$. The Langlands dual group is
$^{L}\SU(r)=\SU(r)/\zz_r$. We make these transformations more precise following \cite[\S 3]{VW}. 
Take $\tau$ as the parameter of the upper half plane $\hh$,  and  let $\Gamma_0(4)\subset \SL(2,\zz)$ be the subgroup
$$\Gamma_0(4)=\left\{
\mat{cc} a&b\\
c&d\rix\in \SL(2,\zz):  4| c\right\}.$$
The group $\Gamma_0(4)$ acts on $\hh$ by
$$\tau\mapsto \frac{a\tau+b}{c\tau+d}.$$ The group $SL(2,\zz)$ is generated by transformations
$$S=\mat{cc} 0&-1\\
1&0\rix;  \quad T=\mat{cc} 1&1\\
0&1\rix.$$
From \cite{VW}, invariance under $T$ is the assertion that physics is periodic in $\theta$ with period $2\pi$, and 
$S$
is equivalent at $\theta=0$ to the transformation $\frac{g^2}{4\pi}\mapsto (\frac{g^2}{4\pi})^{-1}$ originally proposed by
Montonen and Olive \cite{MO}. One can check  that 
$T(\tau)=\tau+1$, and $S(\tau)=-\frac{1}{\tau}$. 

For a smooth projective surface $S$,  let $Z(S, \SU(r);\tau)=Z(S, \SU(r); q)$ be the partition function which counts the invariants of  $\SU(r)$-instanton moduli spaces, where we let $$q=e^{2\pi i \tau}.$$ Similarly let  
$Z(S, \SU(r)/\zz_r; \tau)$ be the  partition function which counts the invariants of $\SU(r)/\zz_r$-instanton moduli spaces.  As pointed out in \cite[\S 3]{VW}, when certain vanishing theorem holds, the invariants count the Euler characteristic of  instanton moduli spaces.  
We have the S-duality conjecture proposed in \cite{VW}, and reviewed in \cite{Jiang_2019}:
\begin{con}\label{con_1}
The transformation $T$ acts on $Z(S,\SU(r); q)$, and the $S$-transformation sends
\begin{equation}\label{eqn_S_transformation}
Z\left(S,\SU(r); -\frac{1}{\tau}\right)=\pm r^{-\frac{\chi}{2}}\left(\frac{\tau}{i}\right)^{\frac{\omega}{2}}Z(S,\SU(r)/\zz_r; \tau).
\end{equation}
for some $\omega$, where  $\chi:=\chi(S)$ is the topological Euler number of $S$.  
\end{con}
Usually $\omega=\chi$.
This is Formula (3.18) in \cite{VW}. 
In mathematics we view $Z(S,\SU(r);\tau)=Z(S,\SU(r);q)$ as the partition function which counts the invariants of  moduli space of vector bundles or Higgs bundles on $S$.  
Then the S-duality conjecture predicts that $T^4$ acts on the $\SU(r)/\zz_r$-theory to itself; and 
$ST^4 S=\mat{cc} 1&0\\
4&1\rix$ will map the $\SU(r)$-theory to itself.   Note that 
$\Gamma_0(4)=\langle T, ST^4 S\rangle$ is generated by $T, ST^4 S$. 
 In the case of a spin manifold, we get the subgroup of $\SL(2,\zz)$ generated by $S$ and $ST^2S$, which is the group 
$$\Gamma_0(2)=\left\{
\mat{cc} a&b\\
c&d\rix\in \SL(2,\zz):  2| c\right\}.$$  
Therefore if the S-duality conjecture holds, the partition  function $Z(S,\SU(r);\tau)=Z(S,\SU(r);q)$ is a modular form with modular group 
$\Gamma_0(4)$ or $\Gamma_0(2)$ if $M$ is a spin manifold.

In \cite[\S 4]{VW}, Vafa-Witten made a prediction   that the S-duality conjecture \ref{con_1} holds for the cases $K3$ surface and $\pp^2$ in rank two.   In the case of K3 surfaces Vafa-Witten \cite{VW} also gave a formula for the partition function for the  
$\SU(r)$-invariants 
for prime rank $r$, and this formula was proved by Tanaka-Thomas \cite{TT2} by calculating the $\SU(r)$-Vafa-Witten invariants.  In \cite{Jiang_2019} the first author provided a rigorous mathematical definition of twisted Vafa-Witten invariants.  These twisted Vafa-Witten invariants were used to  define the Langlands dual gauge group $\SU(r)/\zz_r$-Vafa-Witten invariants and prove the S-duality conjecture 
for $\pp^2$ in rank two and K3 surfaces in all prime ranks. 
To the author's knowledge, there is no formula for the $\SU(r)/\zz_r$-invariants for K3 surfaces in the literature for non prime rank $r$.  In this paper we calculate the  formula for the $\SU(r)/\zz_r$-Vafa-Witten invariants  for K3 surfaces for any rank $r$. Thus   what we do here is actually beyond physicists's prediction.

\subsection{S-duality conjecture using twisted Vafa-Witten invariants}
\label{subsec_S-duality_Jiang_proposal_intro}

In \cite{TT1}, \cite{TT2}, \cite{Thomas2}, Tanaka-Thomas have developed a theory for the gauge group $\SU(r)$-Vafa-Witten invariants for smooth projective surfaces $S$.  Tanaka-Thomas have defined the Vafa-Witten invariants on the moduli spaces of semistable Higgs pairs $(E,\phi)$ where $E$ is a torsion free sheaf of rank $r$, and $\phi: E\to E\otimes K_S$ is a section called the Higgs field. 
Let  $$\N^{s}:=\N^{s}_{S}(r, L, c_2)$$ be the moduli space of stable Higgs sheaves on $S$ with fixed topological data $(r, L, c_2)$. 
$\N^s$  is not compact, but admits a $\cc^*$-action with compact fixed loci.   Tanaka-Thomas have constructed a symmetric obstruction theory on $\N^{s}$ and defined the invariants using virtual localization of \cite{GP}.    The invariant is  denoted by $\VW(S)$, and is called the $\SU(r)$-Vafa-Witten invariant.  The invariant  $\vw(S)$ is defined as the weighted Euler characteristic weighted by the Behrend function  on  $\N^{s}$. In general these two invariants 
$\VW(S)$ and $\vw(S)$ are not the same, but they are equal for K3 surfaces, see \cite{MT}. 

Tanaka-Thomas \cite{TT2}  also defined the Vafa-Witten invariants  $\VW(S)$ and $\vw(S)$  for counting strictly semistable Higgs sheaves using Joyce-Song's stable pair techniques \cite{JS}.  
The invariants  $\vw(S)$  are defined using Joyce-Song's method of wall-crossing, and for  general surfaces the invariants   $\VW(S)$ are defined by conjectures.  \cite{MT} proved that $\VW(S)$ and $\vw(S)$ are the same invariants for K3 surfaces. 
We use the notations  
$\VW(S)$ and $\vw(S)$ to represent the Vafa-Witten invariants for counting semistable Higgs sheaves as in \cite{TT2}.  In \cite[Theorem 1.7]{TT2} Tanaka-Thomas calculated the partition function of the Vafa-Witten invariants $\VW(S)$ for K3 surfaces with any rank $r$ and fixed determinant $\sO$.

S-duality equation (\ref{eqn_S_transformation}) implies that one has to consider the Vafa-Witten invariants for the gauge group $\SU(r)/\zz_r$, which is the Langlands dual of $\SU(r)$.  In mathematics the theory for the dual group 
$\SU(r)/\zz_r$ is the moduli space or stack of $\PGL_r$-bundles or $\PGL_r$-Higgs  bundles.  In \cite{Jiang_2019}, the first author has developed a theory of twisted Vafa-Witten theory using twisted sheaves and twisted Higgs sheaves on a surface $S$.  The moduli stack is constructed on any  cyclic $\mu_n$-gerbe over a smooth surface $S$ for $n\in\zz_{>0}$.   Basic notion of  $\mu_n$-gerbes can be found in \cite{LMB}.  
The notion of essentially trivial $\mu_n$-gerbes and optimal $\mu_n$-gerbes can be found in \cite{Lieblich_Duke}.
If a $\mu_n$-gerbe 
$$\SS\to S$$ 
is optimal, which means that its corresponding Brauer class in the cohomological Brauer group $H^2(S,\sO_S^*)_{\tor}$ is of order $n$,  then  Lieblich   \cite{Lieblich_ANT} proved there is a covering morphism from the moduli stack of stable twisted sheaves on $S$ to the moduli stack of $\PGL_n$-bundles. A Higgs bundle version of this result has been proved in \cite{Jiang_2019-2}. Therefore it is  
promising to using moduli stack of twisted Higgs sheaves to attack the S-duality conjecture. 
It is worth pointing out that for a $\mu_n$-gerbe $\SS\to S$ the category of coherent sheaves $\Coh(\SS)$ has a decomposition 
$\Coh(\SS)=\bigoplus_{0\leq i\leq n-1}\Coh(\SS)_i$, where $\Coh(\SS)_i$ is the subcategory of coherent sheaves on $\SS$ with $\mu_n$-weight $i$.  The category of twisted sheaves is the subcategory $\Coh(\SS)_1$.

In \cite{Jiang_2019}, the first author defined the moduli stack  $$\N^{\tw,s}:=\N^{\tw, s}_{\SS}(r,L,c_2)$$  of stable twisted  Higgs sheaves $(E,\phi)$ on a $\mu_n$-gerbe $\SS\to S$ with fixed 
$\textbf{c}=(r,L,c_2)$, where $E$ is a torsion free twisted sheaf of rank $r$ and $\phi: E\to E\otimes K_{\SS}$ is a section still called the Higgs field.  Then similar to Tanaka-Thomas \cite{TT1}, there is a symmetric obstruction theory on $\N^{\tw,s}$.   \cite{Jiang_2019} has defined two twisted Vafa-Witten invariants 
$\VW^{\tw}_{\SS}(\textbf{c})$ and $\vw^{\tw}_{\SS}(\textbf{c})$, where the first is defined using virtual localization and the second is defined using the Behrend function as the weighted Euler characteristic on the moduli space. 
They are called the $\SU(r)/\zz_r$-Vafa-Witten invariants when $n=r$. 

\cite{Jiang_2019} also constructed the generalized twisted Vafa-Witten invariants 
$\VW^{\tw}_{\SS}(\textbf{c})$ and $\vw^{\tw}_{\SS}(\textbf{c})$ for counting semistable twisted Higgs sheaves 
$(E,\phi)$.  The method is the same as in \cite{TT2} using the Joyce-Song techniques \cite{JS}. 
More details can be found in \cite[\S 5]{Jiang_2019}.  We use $\VW^{\tw}_{\SS}(\textbf{c})$ and $\vw^{\tw}_{\SS}(\textbf{c})$ to represent the generalized twisted Vafa-Witten invariants. The  invariants $\VW^{\tw}_{\SS}(\textbf{c})$ are defined by conjectures, see \cite[Conjecture 5.10]{Jiang_2019}; and the second $\vw^{\tw}_{\SS}(\textbf{c})$ are defined using Joyce-Song wall crossing formula.  In general they are not equal.  In this paper we prove that they are the same invariants for K3 gerbes.  The method we use is similar to \cite{MT} for the K3 case, where we use  a cyclic gerbe version of Oberdieck's theorem in \cite{Oberdieck} which we prove in \cite{Jiang-Tseng_multiple}.  We also need to relate the Oberdieck invariants to Behrend function in \cite{Jiang-Tseng_multiple}.

In \cite{Jiang_2019}, the first author has defined  the $\SU(r)/\zz_r$-Vafa-Witten invariants using the twisted Vafa-Witten invariants for all $\mu_r$-gerbes on a surface $S$. 
The  following conjecture was made in  \cite{Jiang_2019}.  Let $S$ be a smooth projective surface. 

\begin{defn}\label{defn_SU/Zr_VW_intro}
Fix an $r\in\zz_{>0}$, for any $\mu_r$-gerbe $p: \SS_g\to S$
corresponding to $g\in H^2(S,\mu_r)$, let $\sL\in\Pic(\SS_g)$ and 
let 
$$Z_{r,\sL}(\SS_g, q):=\sum_{c_2}\VW^{\tw}_{(r,\sL, c_2)}(\SS_g)q^{c_2}$$
be the generating function of the twisted Vafa-Witten invariants. 

Let us fix a line bundle 
$L\in\Pic(S)$, and define for any essentially trivial $\mu_r$-gerbe $\SS_g\to S$ corresponding to the line bundle $\sL_g\in\Pic(S)$, 
$L_g:=p^*L\otimes \sL_g$; 
for all the other $\mu_r$-gerbe $\SS_g\to S$, keep the same $L_g=p^*L$. 
Also for $L\in\Pic(S)$, let $\overline{L}\in H^2(S,\mu_r)$ be the image under the morphism 
$H^1(S,\sO_S^*)\to H^2(S,\mu_r)$. 

We define
$$Z_{r,L}(S, \SU(r)/\zz_r;q):=\sum_{g\in H^2(S,\mu_r)}e^{\frac{2\pi i g\cdot \overline{L}}{r}}Z_{r,L_g}(\SS_g, q).$$
We call it the partition function of $\SU(r)/\zz_r$-Vafa-Witten invariants. This is parallel to the physics conjecture in \cite[(5.22)]{LL}. 
\end{defn}

\begin{con}\label{con_S_duality_intro}
For a smooth projective surface $S$, the partition function 
of $\SU(r)$-Vafa-Witten invariants 
$$Z_{r,L}(S, \SU(r);q)=\sum_{c_2}\VW_{(r,L, c_2)}(S)q^{c_2}$$ 
 and 
the partition function of $\SU(r)/\zz_r$-Vafa-Witten invariants 
$Z_{r,L}(S, \SU(r)/\zz_r;q)$ satisfy  the S-duality conjecture in equation (\ref{eqn_S_transformation}). 
\end{con}

In \cite[Theorem 1.8]{Jiang_2019} we prove Conjecture \ref{con_S_duality_intro} for projective plane in rank two; and in  \cite[Theorem 1.11]{Jiang_2019} and  \cite[Theorem 6.38]{Jiang_2019} we 
prove conjectural equation (\ref{eqn_S_transformation}) for K3 surfaces first in rank two and then in prime ranks.
In \cite{JK} Jiang and Kool proved the S-duality conjecture for K3 surfaces in prime rank case using only Yoshioka  twisted sheaves not involving the gerbe construction. 
In this paper we generalize and prove this conjecture for any rank $r$ for K3 surfaces.

\subsection{Twisted multiple cover formula}\label{subsec_multiple_cover_twist_intro}

The method to prove Conjecture \ref{con_S_duality_intro} for a K3 surface $S$ in prime rank $r$ is as follows.   The cohomology group  $H^2(S,\mu_r)$ classifies all the $\mu_r$-gerbes over $S$.  For each element $g\in H^2(S,\mu_r)$,  if $g\neq 0$, let $\SS_{\opt}\to S$ be an optimal $\mu_r$-gerbe,  then we define 
$$\Z_{r,\sO}(\SS_g; q):=\sum_{c_2}\VW^{\tw}_{\SS_g}(r,\sO,c_2)q^{c_2}.$$
Let $\Z_{r,\sO}(S,; q)$ be the Tanaka-Thomas partition function of the Vafa-Witten invariants for $S$. 
Then 
define:
$$
Z^{\prime}_{r,\sO}(S, \SU(r)/\zz_r;q):=Z_{r,0}(q)+\sum_{0\neq g\in H^2(S,\mu_r)}\sum_{m=0}^{r-1}e^{\pi i \frac{r-1}{r} m g^2}
Z_{r,\sO}(\SS_{\opt}, (e^{\frac{2\pi i m}{r}})^r\cdot q)
$$
where $Z_{r,\sO}(\SS_{\opt}, (e^{\frac{2\pi i m}{r}})^r\cdot q)$ is a variation of the partition function $Z_{r,\sO}(\SS_{\opt},  q)$. 
We calculate:
$$Z^{\prime}(S, \SU(r)/\zz_r; q)=\frac{1}{r^2}q^r G(q^r)+q^r\left(r^{21} G(q^{\frac{1}{r}})+r^{10} \left(\sum_{m=1}^{r-1}G\left(e^{\frac{2\pi i m}{r}}q^{\frac{1}{r}}\right)\right) \right)$$
which exactly matches the physics prediction in \cite{LL}, see \cite[Theorem 6.38]{Jiang_2019}. Here $G(q):=\eta(q)^{-24}$ and $\eta(q)$ is the Dedekind eta function. 

If $r$ is not a prime integer,  the partition function $\Z_{r,\sO}(S;q)$ of Tanaka-Thomas's Vafa-Witten invariants  $\VW_{r,0,c_2}(S)$ was calculated using Toda's multiple 
cover formula for the counting invariants for semistable coherent sheaves on local K3 surface $$X:=S\times\cc,$$ see \cite[Corollary 6.8, 6.10]{MT}, \cite[Conjecture 1.3]{Toda_JDG}. 
This is because  the Vafa-Witten  invariants counting semistable Higgs sheaves on $S$ are the invariants counting torsion two-dimensional sheaves on the local K3 surface $X$.
For the Langlands dual group $\SU(r)/\zz_r$, in the prime rank $r$ case, any twisted Higgs sheaf $(E,\phi)$ of rank $r$ on an optimal $\mu_r$-gerbe $\SS_{\opt}\to S$ is automatically stable, since counting rank $r$
twisted sheaves on $\SS_{\opt}$ is the non-commutative analogue of Picard scheme on an Azumaya algebra $\sA$ over $S$ corresponding to the optimal $\mu_r$-gerbe $\SS_{\opt}\to S$. Thus in prime rank $r$ case there is no need for the 
multiple cover formula for the twisted sheaves on optimal $\mu_r$-gerbes $\SS_{\opt}$. 

But to get any rank  $\SU(r)/\zz_r$-Vafa-Witten invariants,  we need a multiple cover formula for the counting  invariants of semistable twisted sheaves on the local  optimal K3 gerbe.
Let $S$ be a smooth projective K3 surface.  Let $g\in H^2(S,\mu_r)$ and the order $|g|=o_{\opt}| r$.  An optimal $\mu_{o_{\opt}}$-gerbe $\SS_{\opt}\to S$ on $S$ corresponding to $g$ determines a nontrivial Brauer class 
$\varphi(\SS_{\opt})=\alpha\in H^2(S,\sO_S^*)_{\tor}$
through the exact sequence 
$$\cdots\to H^1(S,\sO_S^*)\to H^2(S,\mu_r)\stackrel{\varphi}{\longrightarrow} H^2(S,\sO_S^*)\to \cdots$$
induced by the short exact sequence:
$$1\to \mu_r\to \sO_S^*\stackrel{(\cdot)^r}{\longrightarrow} \sO_S^*\to 1.$$
The cohomology $H^2(S,\sO_S^*)_{\tor}$ is, by definition, the cohomological Brauer group 
$\Br^\prime(S)$, and from de Jong's theorem \cite{de_Jong} the Brauer group 
$\Br(S)=\Br^\prime(S)$. 
We call $$(S,\alpha),$$ a K3 surface $S$ together with a Brauer class $\alpha\in \Br^\prime(S)$ a twisted K3 surface as in \cite{HS}. 
We also call the optimal gerbe $\SS_{\opt}\to S$ a twisted K3 surface since its class in  $\Br^\prime(S)$ is $\alpha$.  We always use these two notions. 
The twisted Mukai vectors were constructed in \cite{HMS} and \cite{Yoshioka2}. 
Let $\Coh(S,\alpha)$ or $\Coh^{\tw}(\SS_{\opt})$ be the category of twisted sheaves on a twisted K3 surface. 

Let $\XX_{\opt}:=\SS_{\opt}\times\cc$ which is the local K3 gerbe.   It is an optimal $\mu_{o_{\opt}}$-gerbe over $X=S\times\cc$ and its class 
in $H^2(X,\mu_r)\cong H^2(S,\mu_r)$ is also given by $\alpha$. 
We let $\Coh(X,\alpha)$ or $\Coh^{\tw}(\XX_{\opt})$ be the category of twisted sheaves on a local  twisted K3 surface $(X,\alpha)$. 
Let $H(\sA^{\tw}_{\XX_{\opt}})$ be the Hall algebra of the category $\sA_{\XX_{\opt}}^{\tw}=\Coh^{\tw}(\XX_{\opt})$ as in \cite{Bridgeland10}, \cite{JS}. 
We use the geometric stability in \cite{Lieblich_Duke} or the modified stability for a generating sheaf $\Xi$ and a polarization $\sO_X(1)$ for twisted sheaves and the moduli stack construction therein. 
Let $\Gamma_0:=\zz\oplus\NS(S)\oplus\qq$. 
Then there is a virtual indecomposable element 
$$
\epsilon_{\omega,\XX_{\opt}}(v):=
\sum_{\substack{\ell\geq 1, v_1+\cdots+v_{\ell}=v, v_i\in\Gamma_0\\
p_{\omega,v_i}=p_{\omega,v}(m)}}
\frac{(-1)^{\ell}-1}{\ell}\delta_{\omega,\XX_{\opt}}(v_1)\star\cdots\star \delta_{\omega,\XX_{\opt}}(v_{\ell})
$$
where $p_{\omega,v_i}(m)$ is the reduced geometric Hilbert polynomial, and 
$$\delta_{\omega,\XX_{\opt}}(v):=[\sM_{\omega,\XX_{\opt}}(v)\hookrightarrow \widehat{\sM}(\XX_{\opt}) ]\in H(\sA^{\tw}_{\XX_{\opt}})$$
is an element in the Hall algebra.  Here $\omega\in \NS(S)$ is an ample divisor, $\sM_{\omega,\XX_{\opt}}(v)\hookrightarrow \widehat{\sM}(\XX_{\opt}) $ is the moduli stack of semistable twisted sheaves 
with Mukai vector $v$, and $\widehat{\sM}(\XX_{\opt})$ is the stack of coherent twisted sheaves on $\XX_{\opt}$. 
Then the invariants $J(v)$ is defined by the Poincar\'e polynomial of $\epsilon_{\omega,\XX_{\opt}}(v)$.  
The following  multiple cover formula for the invariants $J(v)$ was proved in \cite{Jiang-Tseng_multiple}.

\begin{thm}(\cite{Jiang-Tseng_multiple})\label{thm_twisted_multiple_cover_intro}
We have 
$$J(v)=\sum_{k\geq 1; k|v}\frac{1}{k^2}\chi(\Hilb^{\langle v/k,v/k\rangle+1}(S))$$
\end{thm}

\begin{rmk}
One can understand this twisted multiple cover formula as follows.  From \cite{H-St},  a twisted sheaf on the local K3 gerbe $\XX_{\opt}$ is actually an untwisted sheaf and the Mukai vector $v$ is defined by the Chern character of this untwisted sheaf.  So the invariant $J(v)$ is actually understood as the counting invariant for untwisted sheaves on the local K3 surface $X$.  Thus the multiple cover formula is just the multiple cover formula of Y. Toda. 
\end{rmk}

\subsection{All ranks S-duality for K3 surfaces}\label{subsec_higher_rank_K3}

In this section we provide a calculation for  any  rank S-duality conjecture for K3 surfaces.  For a K3 surface $S$, the twisted Vafa-Witten invariants $\VW^{\tw}_{v}(\SS_{\opt})=\vw^{\tw}_{v}(\SS_{\opt})$ for any 
optimal $\mu_{o_{\opt}}$-gerbe $\SS_{\opt}\to S$ are the invariants $J(v)$ with the Mukai vector $v$. 
In general the invariants $\vw^{\tw}_{v}(\SS_{\opt})$ are the Joyce-Song invariants $\JS^{\tw}_{v}(\SS_{\opt})$
defined by the same element $\epsilon_{\omega,\XX_{\opt}}(v)$ but with the Behrend function on the moduli stack $\widehat{\sM}(\XX_{\opt})$.  Since from \cite[Proposition 5.9]{TT2}, on the $\cc^*$-fixed Higgs sheaves, the Behrend function is always $-1$ (the proof works for twisted Higgs pairs), the invariants  $J(v)$ are the same as the Joyce-Song invariants. 
Therefore we can use the multiple cover formula in Theorem \ref{thm_twisted_multiple_cover_intro} to calculate the partition function for any rank twisted Vafa-Witten invariants.

Let $S$ be a smooth projective K3 surface.  From \cite[Theorem 1.7]{TT2}, for any rank $r\in\zz_{>0}$, the  partition function of the Vafa-Witten invariants $\VW(S)$ was calculated as:
\begin{align}\label{eqn_TT_higher_rank_K3}
\Z_{r,\sO}(S, \SU(r); q)&= \sum_{c_2}\VW_{r,0,c_2}(S)q^{c_2}\\  \nonumber
&=\sum_{e|r}\frac{e}{r^2}q^r\sum_{m=0}^{e-1} G\left(e^{\frac{2\pi i m}{e}}q^{\frac{r}{e^2}}\right)\\ \nonumber
&=q^r\cdot \frac{1}{r^2}\cdot G(r\cdot \tau)+\sum_{\substack{e|r\\e\neq 1}}\frac{e}{r^2}q^r\left( G(\frac{r\tau}{e^2})+ \sum_{m=1}^{e-1} G\left(\frac{m+\frac{r}{e}\tau}{e}\right)\right).
\end{align}
Now using the formula in \cite[Formula A.4]{LL},  (note that there are some typos in the formula) we calculate after the $S$-transformation $\tau\mapsto -\frac{1}{\tau}$,
\begin{equation}\label{eqn_S_transformation1}
\begin{array}{ll}
G(r\cdot \tau)\mapsto & \tau^{-12}r^{12}\cdot G(\frac{\tau}{r});\\ 
G(\frac{r}{e^2}\cdot \tau)\mapsto &  \tau^{-12}\left(\frac{r}{e^2}\right)^{12}\cdot G(\frac{e^2}{r}\tau);\\ 
G\left(\frac{\frac{r}{e}\tau+m}{e}\right)\mapsto &  \tau^{-12}\left(\frac{\overline{d}}{d}\right)^{12}\cdot G\left(\frac{d\tau}{\overline{d} o_{\opt}}+\frac{ds}{d\cdot o_{\opt}}\right);\\
\end{array}
\end{equation}
where in the third formula above we have:
$$1\leq s<o_{\opt}; \quad  
d:=\gcd(m,e); \quad  \overline{d}=\frac{r}{e}$$
and $s$ satisfies 
$$s \cdot \frac{m}{d}\equiv -1\mod (\frac{e}{d})$$
and we let 
$o_{\opt}:=\frac{e}{d}$. 
 We calculate that for all the  $m$'s satisfying  $1\leq m\leq e-1$ and  $d=\gcd(m,e)$, if $d=1$, then $\frac{ds}{d\cdot o_{\opt}}$ will take all the values $\frac{s}{e}$
with $\gcd(s,e)=1$; and if $d>1$, then $\frac{ds}{d\cdot o_{\opt}}$ will take all the values 
with $\gcd(ds,e)=d$, i.e., $s$ takes all the values in $[1,o_{\opt}-1]$.

Therefore after the $S$-transformation $\tau\mapsto -\frac{1}{\tau}$, $\Z_{r,\sO}(S, \SU(r); q)$ becomes:
\begin{align}\label{eqn_above_1}
&q^r\cdot \frac{1}{r^2}\cdot \Big[ \tau^{-12}r^{12}\cdot G\left(\frac{\tau}{r}\right) \Big]+
\sum_{\substack{e|r\\e\neq 1}}
\Big[\frac{e}{r^2}q^r\cdot  \tau^{-12}\left(\frac{r}{e^2}\right)^{12}\cdot G\left(\frac{e^2}{r}\tau\right) \\ \nonumber
&+ \sum_{\substack{1\leq m\leq e-1,\\
 (m,e)=1}}\frac{e}{r^2}q^r\cdot  \tau^{-12}\overline{d}^{12}\cdot G\left(\frac{\tau}{r}+\frac{e-m}{e}\right)+
 \sum_{\substack{1\leq m\leq e-1,\\
  (m,e)=d>1}}\frac{e}{r^2}q^r\cdot  \tau^{-12}\left(\frac{\overline{d}}{d}\right)^{12}\cdot G\left(\frac{d\tau}{\overline{d} o_{\opt}}+\frac{ds_m}{d\cdot o_{\opt}}\right)
\Big]
\end{align}
where $s_m$ is the corresponding $s$ in (\ref{eqn_S_transformation1}). We may write it in more detail:
\begin{align*}
&q^r\cdot \frac{1}{r^2}\cdot \Big[ \tau^{-12}r^{12}\cdot G\left(\frac{\tau}{r}\right) \Big]+
\sum_{\substack{e|r\\e\neq 1, e\neq r}}
\Big[\frac{e}{r^2}q^r\cdot  \tau^{-12}\left(\frac{r}{e^2}\right)^{12}\cdot G\left(\frac{e^2}{r}\tau\right) \\
&+ \sum_{\substack{1\leq m\leq e-1,\\
 (m,e)=1}}\frac{e}{r^2}q^r\cdot  \tau^{-12}\overline{d}^{12}\cdot G\left(\frac{\tau}{r}+\frac{e-m}{e}\right)+
 \sum_{\substack{1\leq m\leq e-1, \\
  (m,e)=d>1}}\frac{e}{r^2}q^r\cdot  \tau^{-12}\left(\frac{\overline{d}}{d}\right)^{12}\cdot G\left(\frac{d\tau}{\overline{d} o_{\opt}}+\frac{ds_m}{d\cdot o_{\opt}}\right)
\Big]\\
&+
\Big[\frac{1}{r}q^r\cdot  \tau^{-12}\left(\frac{1}{r}\right)^{12}\cdot G\left(r\tau\right) 
+ \sum_{\substack{1\leq m\leq r-1,\\
 (m,r)=1}}\frac{1}{r}q^r\cdot  \tau^{-12}\cdot G\left(\frac{\tau}{r}+\frac{r-m}{r}\right)\\
&+
 \sum_{\substack{1\leq m\leq r-1,\\
 (m,r)=d>1}}\frac{1}{r}q^r\cdot  \tau^{-12}\left(\frac{1}{d}\right)^{12}\cdot G\left(\frac{d\tau}{o_{\opt}}+\frac{s_m}{ o_{\opt}}\right)\Big]
\end{align*}
where the last sum in the bracket $\Big[,\Big]$ is the term when $e=r$ in (\ref{eqn_above_1}). 

So from conjectural equation (\ref{eqn_S_transformation}),  we modify here a bit, and let $o_{\opt}|r$ in the above formula represent the order of the group element 
in $H^2(S,\mu_r)=\zz_r^{22}$, and in the S-duality conjecture this represents the group $\SU(o_{\opt})/\zz_{o_{\opt}}$,  we 
multiply $o_{\opt}^{11}$. In  the terms involving $m$ such that $(m,e)=d>1$, we multiply $(\frac{e}{\overline{d}})^{11}$, and for  all the other terms we multiply $r^{11}$.   We have the predicted formula for $\Z^\prime_{r,\sO}(S, \SU(r)/\zz_r; q)$:
(we use the notation $\Z^\prime_{r,\sO}(S, \SU(r)/\zz_r; q)$ in order to distinguish from the partition function $\Z_{r,\sO}(S, \SU(r)/\zz_r; q)$ defined by all $\mu_r$-gerbes later.)
\begin{align}\label{eqn_SU(r)_Z_r_formula_K3_intro}
&\Z^\prime_{r,\sO}(S, \SU(r)/\zz_r; q)= \\ \nonumber
&\Big[q^r\cdot r^{21}\cdot G\left(q^{\frac{1}{r}}\right) \Big]+
\sum_{\substack{e|r\\e\neq 1, e\neq r}}
\Big[q^r\cdot \left(\frac{r^{21}}{e^{21}}\right)\cdot \frac{1}{e^2}\cdot G\left(q^{\frac{e^2}{r}}\right) \\ \nonumber
&+ \sum_{\substack{1\leq m\leq e-1,\\
 (m,e)=1}}q^r\cdot r^{10}\cdot G\left(e^{2\pi i\frac{e-m}{e}} q^{\frac{1}{r}}\right)+
 \sum_{\substack{1\leq m\leq e-1,
\\ (m,e)=d>1\\
 e=d\cdot o_{\opt}}}q^r\cdot o_{\opt}^{10}\cdot\frac{1}{d^2}\cdot \frac{1}{\overline{d}}\cdot  G\left(e^{2\pi i\frac{s}{o_{\opt}}} q^{\frac{d}{\overline{d}\cdot o_{\opt}}}\right)\Big] \\ \nonumber
&+\Big[q^r\cdot \frac{1}{r^2} G(q^{r})
+ \sum_{\substack{1\leq m\leq r-1, \\
(m,r)=1}}q^r\cdot  r^{10}\cdot G\left(e^{2\pi i\frac{r-m}{r}} q^{\frac{1}{r}}\right)+
 \sum_{\substack{1\leq m\leq r-1,\\
  (m,r)=d>1}}q^r\cdot  \left(\frac{r}{d}\right)^{10}\cdot \frac{1}{d^2}\cdot G\left(e^{2\pi i\frac{s}{o_{\opt}}} q^{\frac{r}{o^2_{\opt}}}\right)\Big]
\end{align}

Note that if $d\cdot e=r$, then $\frac{d^2}{r}=\frac{r}{e^2}$.
We write further (combining the terms involving $q^{\frac{1}{r}}$ in the above three rows) for (\ref{eqn_SU(r)_Z_r_formula_K3_intro}):
\begin{align}\label{eqn_SU(r)_Z_r_formula_K3_2_intro}
&\Z^\prime_{r,\sO}(S, \SU(r)/\zz_r; q)=\Big[q^r\cdot \frac{1}{r^2} G(q^{r})\Big]+ q^r\cdot r^{21} \cdot G(q^{\frac{1}{r}})+q^{r}\cdot r^{10}\cdot \sum_{m=1}^{r-1} G\left(e^{2\pi i \frac{m}{r}}\cdot q^{\frac{1}{r}}\right)
 \\ \nonumber
&+
\sum_{\substack{d|r\\d\neq 1, d\neq r}}
\Big[q^r\cdot \left(\frac{r^{21}}{d^{21}}\right)\cdot \frac{1}{d^2}\cdot G\left(q^{\frac{r}{o_{\opt}^2}}\right)\Big] \\ \nonumber
&+
\sum_{\substack{e|r\\
e\neq 1,e\neq r}} \Big[\sum_{\substack{1\leq m\leq e-1,\\
 (m,e)=d>1}}q^r\cdot  o_{\opt}^{10}\cdot\frac{1}{d^2}\cdot \frac{1}{\overline{d}}\cdot G\left(e^{2\pi i\frac{s}{o_{\opt}}} q^{\frac{d}{\overline{d}\cdot o_{\opt}}}\right)\Big] \\ \nonumber
&+
 \sum_{\substack{1\leq m\leq r-1,\\
 (m,r)=d>1}}q^r\cdot  \left(\frac{r}{d}\right)^{10}\cdot \frac{1}{d^2}\cdot G\left(e^{2\pi i\frac{s}{o_{\opt}}} q^{\frac{r}{o^2_{\opt}}}\right)
\end{align}
where in $\sum_{\substack{d|r\\d\neq 1, d\neq r}}
q^r\cdot \left(\frac{r^{21}}{d^{21}}\right)\cdot \frac{1}{d^2}\cdot G\left(q^{\frac{r}{o_{\opt}^2}}\right)$ we count all 
$d\cdot o_{\opt}=r$.

\begin{rmk}
If $r$ is a prime number, then we have $e=1$ or $r$, we calculate after the $S$-transformation $\tau\mapsto -\frac{1}{\tau}$,
$$
\begin{array}{ll}
G(r\cdot \tau)\mapsto & \tau^{-12}r^{12}\cdot G(\frac{\tau}{r});\\ 
G(\frac{1}{r}\cdot \tau)\mapsto &  \tau^{-12}\left(\frac{1}{r}\right)^{12}\cdot G(r \tau);\\ 
G\left(\frac{\tau+m}{r}\right)\mapsto &  \tau^{-12}\cdot G\left(\frac{\tau+h}{r}\right), 1\leq m\leq r-1, mh\equiv -1\mod r;\\
\end{array}
$$
which is induced  from (\ref{eqn_S_transformation1}). 
We have:
$$\Z^\prime_{r,\sO}(S, \SU(r)/\zz_r; q)=\frac{1}{r^2}q^r G(q^r)+q^r\left(r^{21} G(q^{\frac{1}{r}})+r^{10} \left(\sum_{j=1}^{r-1}G\left(e^{\frac{2\pi i j}{r}}q^{\frac{1}{r}}\right)\right) \right).
$$
which is exactly Theorem 6.38 in \cite{Jiang_2019}. 
Therefore if the  rank $r$  is not prime,  the partition function $\Z^\prime_{r,\sO}(S, \SU(r)/\zz_r; q)$ has three extra terms in (\ref{eqn_SU(r)_Z_r_formula_K3_2_intro}).
These three terms were not known in physics. 
\end{rmk}

Our main result for the twisted Vafa-Witten invariants is:

\begin{thm}\label{thm_higher_rank_K3_intro}(Theorem \ref{thm_SUr/Zr_Vafa-Witten_K3})
Let $S$ be a smooth projective complex K3 surface and $r\in\zz_{>0}$.  We define $\Z^\prime_{r,\sO}(S, \SU(r)/\zz_r; q)$ as:
\begin{align*}
Z^{\prime}_{r,\sO}(S, \SU(r)/\zz_r;q):=Z_{r,0}(q)+\sum_{\substack{o_{\opt}|r, \\
o_{\opt\neq 1}}}\sum_{\substack{0\neq g\in H^2(S,\mu_{o_{\opt}})}}\sum_{m=0}^{o_{\opt}-1}e^{\pi i \frac{o_{\opt}-1}{o_{\opt}} m g^2}
Z^\prime_{r,\sO}(\SS_{\opt}, (e^{\frac{2\pi i m}{o_{\opt}}})^{o_{\opt}}\cdot q)
\end{align*}
where $Z_{r,0}(q)$ is the  Tanaka-Thomas partition function for $g=0\in H^2(S,\mu_r)$ and 
$Z^\prime_{r,\sO}(\SS_{\opt}, (e^{\frac{2\pi i m}{o_{\opt}}})^{o_{\opt}}\cdot q)$ is defined as  part of the  partition function of the twisted Vafa-Witten invariants 
for the optimal $\mu_{\opt}$-gerbe $\SS_{\opt}\to S$, see Definition \ref{defn_Zprime_opt}; and  
$g\in H^2(S,\mu_r)$ such that $|g|=o_{\opt}$, then $\Z^\prime_{r,\sO}(S, \SU(r)/\zz_r; q)$ has the formula (\ref{eqn_SU(r)_Z_r_formula_K3_2_intro}). 
Thus we generalize and prove the S-duality conjecture of Vafa-Witten in any rank for K3 surfaces. 
\end{thm}

\subsection{Outline} 

After the detail explanation of the S-duality conjecture and twisted Vafa-Witten invariants in the intrudction,  in 
\S \ref{sec_twisted_VW_S-duality_K3} we use twisted multiple cover formula to calculate any rank twisted Vafa-Witten invariants for K3 surfaces, and prove the S-duality conjecture.

\subsection{Convention}
We work over complex number  $\cc$   throughout of the paper.    We use Roman letter $E$ to represent a coherent sheaf on a projective DM stack or an \'etale gerbe  $\SS$, and use curly letter $\sE$ to represent the sheaves on the total space 
Tot$(\sL)$ of a line bundle $\sL$ over $\SS$. 
We reserve {\em $\rk$} for the rank of the torsion free coherent sheaves $E$, and when checking the S-duality for $\SU(r)/\zz_r$, $r=\rk$.
We keep the convention in \cite{VW} to use $\SU(r)/\zz_r$ as the Langlands dual group of $\SU(r)$.

%%% ----------------------------------------------------------------------
\subsection*{Acknowledgments}

Y. J. would like to thank  Kai Behrend,  Amin Gholampour, Martijn Kool,  and   Richard Thomas for valuable discussions on the Vafa-Witten invariants, and Yukinobu Toda for the discussion of the multiple cover formula for K3 and twisted K3 surfaces.  Y. J. is partially supported by  NSF DMS-1600997. H.-H. T. is supported in part by Simons foundation collaboration grant.

%%%----------------------------------------------------------------------

%%%%%%%%%%%%%%%%

\section{Proof of the S-duality conjecture Theorem \ref{thm_higher_rank_K3_intro}}\label{sec_twisted_VW_S-duality_K3}

In this section we apply the multiple cover formula for twisted sheaves on twisted K3 surface to calculate any rank twisted Vafa-Witten invariants for K3 gerbes and the $\SU(r)/\zz_r$-Vafa-Witten invariants.  Thus we prove the Vafa-Witten S-duality conjecture  for K3 surfaces in any rank. 

\subsection{Twisted Vafa-Witten invariants counting semistable Higgs pairs}\label{subsec_semistable_twisted_Higgs_pair}

The theory of twisted Vafa-Witten invariants has been constructed in \cite{Jiang_2019} for $\mu_r$-gerbes $\SS\to S$.  We review the construction of counting invariants of
geometric semistable twisted Higgs sheaves for K3 surfaces and use the invariants to calculate the higher rank twisted Vafa-Witten invariants. 
We fix 
$$\XX=\SS\times\cc=\Tot(K_{\SS})\to X$$
which is a $\mu_r$-gerbe over $X=S\times\cc=\Tot(K_S)$ with class $[\XX]=\alpha\in H^2(X,\mu_r)\cong H^2(S,\mu_r)$.  We fix a generating sheaf $\Xi$ on $\SS$ such that 
$\pi^*\Xi$ (still denote it by $\Xi$) is a generating sheaf on $\XX$, where $\pi: \XX\to \SS$ is the projection.  We use the Gieseker (semi)stability of twisted sheaves using the generating sheaf 
$\Xi$, which is equivalent to the geometric stability of Lieblich for twisted sheaves, see \cite{Jiang_2019}.

Let $(E,\phi)$ be a $\SS$-twisted Higgs sheaf on a $\mu_r$-gerbe $\SS\to S$. From \cite[Proposition 2.18]{JP}, there is a spectral $\XX$-twisted sheaf
$\sE_{\phi}$ on $\XX=\Tot(K_{\SS})$ with respect to the polarization $\sO_X(1)=\pi^*\sO_S(1)$.
Also  the Gieseker (semi)stability of the twisted Higgs pair $(E,\phi)$ is equivalent to the Gieseker 
 (semi)stability of $\sE_{\phi}$.  Then we have the Hall algebra 
  $H(\sA^{\tw})$, where $\sA^{\tw}=\Coh^{\tw}_{\pi}(\XX)$ is the category of twisted sheaves.  
  
We review a bit for the definition of the invariants $\JS^{\tw}(v)$.  We are interested in the elements:
$$\mathbb{1}_{\N^{ss,\tw}_{\textbf{c}}}:   \N^{ss,\tw}_{\textbf{c}}\hookrightarrow \Higg^{\tw}_{K_{\SS}}(\SS)\cong \Coh_c^{\tw}(\XX),$$
where $\N^{ss,\tw}_{\textbf{c}}$ is the stack of Gieseker semistable $\XX$-twisted Higgs sheaves $(E,\phi)$ of class $\textbf{c}\in K_0(\XX)$, and 
$\mathbb{1}_{\N^{ss,\tw}_{\textbf{c}}}$ is the inclusion into the stack of all twisted Higgs pairs on $\SS$. 
We consider its ``logarithm":
\begin{equation}\label{eqn_epsilon_alpha}
\epsilon(\textbf{c}):=\sum_{\substack{\ell\geq 1, (\textbf{c}_i)_{i=1}^{\ell}: \textbf{c}_i\neq 0, \forall i\\
p_{\textbf{c}_i}=p_{\textbf{c}},\sum_{i=1}^{\ell}\textbf{c}_i=\textbf{c}}} \frac{(-1)^{\ell}}{\ell}\mathbb{1}_{\N_{\textbf{c}_1}^{ss,\tw}}*\cdots *\mathbb{1}_{\N_{\textbf{c}_\ell}^{ss,\tw}}.
\end{equation}
The key point is that $\epsilon(\textbf{c})$ is virtually indecomposable in the Hall algebra, see 
\cite[Theorem 8.7]{Joyce03}. There is a proof using operators on inertia stacks, see \cite{BP}.
Then the  $\epsilon(\textbf{c})$ is a stack function with algebra stabilizers,
$$\epsilon(\textbf{c})\in \overline{\text{SF}}^{\ind}_{\text{al}}(\Coh^{\tw}_c(\XX),\qq)$$
in Joyce's notation.  From \cite[Proposition 3.4]{JS}, $\epsilon(\textbf{c})$ can be written as 
$$\sum_i a_i\cdot Z_i\times B\Gm$$ 
where $a_i\in \qq$, and $B\Gm$ is the classifying stack.  After removing the factor $B\Gm$ one can pull back the Behrend function to write:
\begin{equation}\label{eqn_epsilon_alpha_stack_function}
\epsilon(\textbf{c}):=\sum_{i}a_i\left(f_i: Z_i\times B\Gm\to \Coh_c^{\tw}(\XX)\right).
\end{equation}
Then \cite[Eqn 3.22]{JS} defines generalized Donaldson-Thomas invariants 
\begin{equation}\label{eqn_generalized_JSalpha}
\JS^{\tw}_{\textbf{c}}(\XX):=\sum_{i} a_i \chi(Z_i, f_i^*\nu)\in \qq
\end{equation}
where $\nu: \sM^{\tw}_{\XX/\kappa}\to \zz$ is the Behrend function.  This invariants are defined by applying the integration map 
to the element $\epsilon(\textbf{c})$ in (\ref{eqn_epsilon_alpha}).  More details of the integration map can be found in \cite{Joyce07}, \cite{JS} and \cite{Bridgeland10}.  
For DM stacks, see \cite{Jiang_DT_Flop}.

As explained in \cite[\S 2]{TT2}, the $\Gm$-action on $\XX\to \SS$ induces an action on the moduli stack 
$\sM^{\tw}_{\XX/\kappa}$ of twisted sheaves on $\XX$.  The $\Gm$-action can be extended to the Hall algebra $H(\sA^{\tw})$, and 
hence $\epsilon(\textbf{c})$ carries a $\Gm$-action covering the one on 
$\Coh_c^{\tw}(\XX)$. Also in the decomposition
$$\epsilon(\textbf{c}):=\sum_{i} a_i\left( f_i: Z_i\times B\Gm\to \Coh_c^{\tw}(\XX)\right),$$
the $Z_i$'s admit $\Gm$-action so that they are $\Gm$-equivariant and the proof uses Kresch's stratification of finite type algebraic stacks
with affine geometric stabilizers. Then the Behrend function is also $\Gm$-equivariant, and 
\begin{equation}\label{eqn_generalized_JSalpha_Gm}
\JS^{\tw}_{\textbf{c}}(\XX)=(\JS^{\tw}_{\textbf{c}}(\XX))^{\Gm}(\XX)=\sum_{i} a_i \chi(Z_i^{\Gm}, f_i^*\nu|_{Z_i^{\Gm}})\in \qq
\end{equation}

\subsubsection{Joyce-Song twisted stable pairs}\label{subsubsec_Joyce-Song_twisted_pair}

We need to use Joyce-Song pair for twisted sheaves in \cite{Jiang_2019}.  
For the $\mu_r$-gerbe $p: \XX\to X$, we fix a minimal rank $[\XX]$-twisted locally free  sheaf $\xi$ such that $\Xi=\oplus_{i=0}^{r-1}\xi^{\otimes i}$ is a generating sheaf for $\XX\to X$.  
A generating sheaf $\Xi$ on $\XX$ is $p$-very ample and contains every irreducible representations of the cyclic group $\mu_r$.
Recall that 
the functor $p_*:  \Coh(\XX)\to \Coh(X)$ is exact. 

Fixing a $K$-group class $\textbf{c}\in K_0(\XX)\cong K_0(\SS)$, for $m>>0$, a Joyce-Song twisted pair consists of the following data
\begin{enumerate}
\item a compactly supported $\XX$-twisted coherent sheaf $\sE$ with $K$-group class $\textbf{c}\in K_0(\XX)$, and \\
\item a nonzero morphism  $s: p_*\Xi\to p_* \sE(m)$.
\end{enumerate}

\begin{defn}\label{defn_stable_JS_twisted_pair}
A Joyce-Song twisted pair $(\sE,s)$ is stable if and only if 
\begin{enumerate}
\item $\sE$ is modified  Gieseker semistable with respect to $(\Xi, \sO_S(1))$.\\
\item if $\sF\subset \sE$ is a proper subsheaf which destabilizes $\sE$, then $s$ does not factor through $p_*\sF(m)\subset p_*\sE(m)$.
\end{enumerate}
Such a pair is denoted by $I^\bullet=\{\Xi(-m)\to \sE\}$.
\end{defn}

For any $\XX$-twisted coherent sheaf 
$\sE$ on $\XX$, we have $H^{\geq 1}(p_*\sE(m))=0$  for $m>>0$.     We then take 
$\sP^{\tw}:=\sP^{\tw}_{\textbf{c}=(\rk,L,c_2)}(\XX)$ to be the moduli stack of $\XX$-twisted stable Joyce-Song pairs 
$$I^\bullet=\{\Xi(-m)\to \sE\}$$
on $\XX$.  
This moduli stack exists since it is the moduli stack of simple complexes on the Deligne-Mumford stack $\XX$, see \cite{Andreini}, \cite{JS}.
The moduli stack 
$\sP^{\tw}$ is not compact, but admits a symmetric obstruction theory since we can take it as the moduli space of simple complexes on a Calabi-Yau threefold Deligne-Mumford stack $\XX$. 
The invariant we use is Behrend's weighted Euler characteristic 
$\widetilde{\sP}^{\tw}_{(\rk,L,c_2)}(m):=\chi(\sP^{\tw}, \nu_{\sP})$.  Under the $\Gm$-action:
$$\widetilde{\sP}^{\tw}_{(\rk,L,c_2)}(m)=\widetilde{\sP}^{\tw,\Gm}_{(\rk,L,c_2)}(m)=\chi\left(\sP^{\Gm}, \nu_{\sP}|_{\sP^{\Gm}}\right).$$
For generic polarization, Joyce-Song invariants $\JS^{\tw}_{\textbf{c}}(\XX)\in \qq$ satisfy the following identities \cite[Theorem 5.27]{JS}:
\begin{equation}\label{eqn_Joyce-Song_wall_crossing}
\widetilde{\sP}^{\tw}_{(\rk,L,c_2)}(m)=\sum_{\substack{\ell\geq 1, (\textbf{c}_i=\delta_i \textbf{c})_{i=1}^{\ell}: \\
\delta_i>0,\sum_{i=1}^{\ell}\delta_i=1}} \frac{(-1)^{\ell}}{\ell !}\prod_{i=1}^{\ell}
(-1)^{\chi(\textbf{c}_i(m))}\cdot \chi(\textbf{c}_i(m))\cdot \JS^{\tw}_{\textbf{c}_i}(\XX).
\end{equation}
Since $\widetilde{\sP}^{\tw}_{(\rk,L,c_2)}(m)$ is deformation invariant, $\JS_{\textbf{c}}^{\tw}$ is also deformation invariant. 
For general 
$\sO_{S}(1)$, the wall crossing formula is complicated as in \cite[Theorem 5.27]{JS}.  
If semistability coincides with stability for twisted sheaves $\sE$, the moduli stack $\sP^{\tw}_{(\rk,L,c_2)}$ is a $\pp^{\chi(\textbf{c}(m))-1}$-bundle over the moduli stack 
$\N^{\tw}_{(\rk, L, c_2)}$ of $\XX$-twisted torsion sheaves $\sE$.  Here $\chi(\textbf{c}(m))=\dim(\Hom(\Xi, \sE(m)))$.
The Behrend function $\nu_{\sP}$ of $\sP^{\tw}_{(\rk,L,c_2)}(\XX)$  is the pullback of the Behrend function on 
$\N^{\tw}_{(\rk, L, c_2)}$ multiplied by $(-1)^{\chi(\textbf{c}(m))-1}$. Therefore: 
\begin{equation}\label{eqn_Joyce-Song_wall_crossing_stable}
\widetilde{\sP}^{\tw}_{(\rk,L,c_2)}(m)=(-1)^{\chi(\textbf{c}(m))-1}\cdot \chi(\textbf{c}(m))\cdot \widetilde{\vw}_{(\rk, L, c_2)}(\SS)
\end{equation}
which is the first term $\ell=1$ in (\ref{eqn_Joyce-Song_wall_crossing}). In general the formula expresses $\widetilde{\sP}^{\tw}_{(\rk,L,c_2)}(m)$ in terms of rational corrections 
from semistable $\XX$-twisted sheaves $\sE$ on $\XX$. 

\subsubsection{Generalized twisted $\SU(\rk)$-Vafa-Witten $\vw$-invariants}

For the $\mu_r$-gerbe $\SS\to S$, if $h^{0,1}(S)>0$, then $\widetilde{\vw}^{\tw}$ defined before vanish because of the action of $\Jac(S)$.
We follow from \cite{TT2} to fix the determinant $\det(E)$. 

Let us fix a line bundle $L\in \Pic(\SS)$, and use the map:
$$\Coh_c^{\tw}(\XX)\stackrel{\det\circ \pi_*}{\longrightarrow}\Pic(\SS).$$
Let us define 
$$\Coh_c^{\tw}(\XX)^{L}:=(\det\circ \pi_*)^{-1}(L).$$
We can restrict Joyce's stack function to this restricted category.  
For any stack function $F:=\left(f: U\to \Coh_c^{\tw}(\XX)\right)$ we define:
$$F^{L}=\left(f: U\times_{\tiny\Coh_c^{\tw}(\XX)}\Coh_c^{\tw}(\XX)^{L}\to \Coh_c^{\tw}(\XX)\right)$$
which is $\mathbb{1}_{\tiny\Coh_c^{\tw}(\XX)^{L}}\cdot F$, where $\cdot$ is the ordinary product \cite[Definition 2.7]{JS}.
Then $\epsilon(\alpha)$ in (\ref{eqn_epsilon_alpha_stack_function}) becomes:
\begin{equation}\label{eqn_epsilon_alpha_L}
\epsilon(\textbf{c})^L:=\sum_{i} a_i\left( f_i: Z_i\times_{\tiny\Coh_c^{\tw}(\XX)} \Coh_c^{\tw}(\XX)^{L}/\Gm\to \Coh_c^{\tw}(\XX)\right).
\end{equation}
Then applying the integration map again and we get the fixed determinant generalized Donaldson-Thomas invariants:
$$\JS^{L,\tw}_{\textbf{c}}(\XX):=\sum_{i}a_i \chi\left(Z_i\times_{\tiny\Coh_c^{\tw}(\XX)} \Coh_c^{\tw}(\XX)^{L}, f_i^*\nu \right).$$
We also compute it using localization:
\begin{equation}\label{eqn_localized_JS_twisted}
\JS^{L,\tw}_{\textbf{c}}(\XX)=(\JS^{L,\tw}_{\textbf{c}})^{\Gm}(\XX)=\sum_{i}a_i \chi\left(Z^{\Gm}_i\times_{\tiny\Coh_c^{\tw}(\XX)} \Coh_c^{\tw}(\XX)^{L}, f_i^*\nu \right).
\end{equation}

\begin{defn}\label{defn_generalized_twisted_vw_L}
We define the twisted $\SU(\rk)$ generalized Vafa-Witten invariants as:
$$\vw^{\tw}_{(\rk, L,c_2)}(\SS):=(-1)^{h^0(K_{\SS})}\JS^{L,\tw}_{(\rk,L,c_2)}(\XX)\in \qq.$$
\end{defn}

As explained in \cite[\S 4]{TT2}, we did not restrict to the trace zero $\tr\phi=0$, but in the $\SU(\rk)$-moduli stack case we did.  The difference  is only a Behrend function sign $(-1)^{\dim H^0(K_{\SS})}$. 
If $h^{0,1}(S)=0$ and $\sO_S(1)$ is generic, then the wall crossing formula (\ref{eqn_Joyce-Song_wall_crossing}) becomes, see (\cite[Formula (5.35)]{Jiang_2019})
\begin{equation}\label{eqn_Joyce-Song_wall_crossing_L}
\widetilde{\sP}^{\tw}_{\textbf{c}}(m)=\sum_{\substack{\ell\geq 1, (\textbf{c}_i=\delta_i \textbf{c})_{i=1}^{\ell}: \\
\delta_i>0,\sum_{i=1}^{\ell}\delta_i=1}} \frac{(-1)^{\ell}}{\ell !}\prod_{i=1}^{\ell}
(-1)^{\chi(\textbf{c}_i(m))+h^0(K_{\SS})}\cdot \chi(\textbf{c}_i(m))\cdot \vw^{\tw}_{\textbf{c}_i}(\SS).
\end{equation}
If $h^{0,1}>0$, and $\sO_S(1)$ generic, we should use Joyce-Song pairs $(\sE,s)$ with fixed determinant
$$\det(\pi_*\sE)\cong L\in\Pic(\SS).$$
Joyce-Song use the category $\sB_{p_{\alpha}}$ before, and we apply the integration map
$\widetilde{\Psi}^{\sB_{p_{\alpha}}}$ to:
$$\epsilon^L_{\textbf{c},1}:=\overline{\epsilon}_{\textbf{c},1}\times_{\tiny\Coh_c^{\tw}(\XX)}\Coh_c^{\tw}(\XX)^L$$
where $\overline{\epsilon}_{\textbf{c},1}$ is the stack function (13.25), (13.26) in \cite[Chapter 13]{JS}.  
Since $\overline{\epsilon}_{\textbf{c},1}$ is virtually indecomposable, we get 
$$\sP^{\tw}_{\textbf{c}=(\rk,L,c_2)}(m):=\chi(\sP^{\tw}, \nu_{\sP}),$$
where $\sP^{\tw}:=\sP^{\tw}_{\textbf{c}}$ is the moduli stack of stable Joyce-Song twisted pairs $(\sE,s)$ with $\det(\pi_*\sE)=L$. 
Hence 
\begin{equation}\label{eqn_Palpha_weighted_Euler}
\sP^{\tw}_{\textbf{c}}(m)=\sP^{\Gm}_{\textbf{c}}(m)=\chi((\sP^{\tw})^{\Gm}_{\textbf{c}}, \nu_{\sP}|_{\sP^{\Gm}_{\textbf{c}}}).
\end{equation}
For  generic $\sO_S(1)$, we get (\cite[Proposition 4.4]{TT2}): if $h^{0,1}(S)>0$, the invariants $(\sP^{\tw})^{\Gm}_{\textbf{c}}(m)$ determines the twisted 
$\SU(\rk)$-Vafa-Witten invariants $\vw^{\tw}$ of Definition \ref{defn_generalized_twisted_vw_L} by:
\begin{equation}\label{eqn_Palpha_L_stable_case}
\sP^{\tw}_{\textbf{c}}(m)=(-1)^{h^0(K_{\SS})}(-1)^{\chi(\textbf{c}(m))-1}\chi(\textbf{c}(m))\vw^{\tw}_{\textbf{c}}(\SS).
\end{equation}
\begin{rmk}
The proof of (\ref{eqn_Palpha_L_stable_case}) is the same as \cite[Proposition 4.4]{TT2}, and the basic reason is that if $h^{0,1}(S)>0$, then the Jacobian $\Jac(S)$ is an abelian variety which acts on the moduli stacks to force the Euler characteristic to be zero. 
\end{rmk}

\subsection{Generalized twisted $\SU(\rk)$-Vafa-Witten invariants $\VW^{\tw}$}\label{subsec_generalized_twisted_VW}

Fix a $\mu_r$-gerbe $\SS\to S$, and the $\mu_r$-gerbe $\XX\to X$.
In this section we generalize the arguments in \cite[\S 6]{TT2} to twisted generalized $\SU(\rk)$-Vafa-Witten invariants $\VW^{\tw}$. 

For $m>>0$, fixing $\textbf{c}=(\rk,L,c_2)$, $L\in\Pic(\SS)$, the Joyce-Song twisted pair $(\sE, s)$ is defined in \S \ref{subsubsec_Joyce-Song_twisted_pair}, where $\sE$ is a $\XX$-twisted sheaf on $\XX$ corresponding to a $\SS$-twisted Higgs sheaf $(E,\phi)$ on $\SS$. Let 
$$\sP^{\perp, \tw}_{\textbf{c}}\subset \sP^{\tw}_{\textbf{c}}$$
be the moduli stack of twisted stable pairs with $\det(\pi_*\sE)=L$, $\tr\phi=0$. 
On $\sP^{\tw}_{\textbf{c}}$, there is a symmetric obstruction theory, see \cite[Chapter 12]{JS}. The tangent-obstruction complex is given by:
$R\cHom_{\XX}(I^\bullet, I^{\bullet})_0[1]$ at the point $I^{\bullet}=\{\Xi\stackrel{s}{\longrightarrow}\sE\}$.  By \cite[\S 5]{TT1}, or \cite[Appendix]{JP} we can modify it to a symmetric obstruction theory on 
$\sP^{\perp, \tw}_{\textbf{c}}$.  

We apply the $\Gm$-virtual localization \cite{GP} to define:
\begin{defn}\label{defn_JS_twisted_pair_invaraint}
$$\sP^{\perp,\tw}_{\textbf{c}}(m):=\int_{[(\sP^{\perp,\tw}_{\textbf{c}})^{\Gm}]^{\vir}}\frac{1}{e(N^{\vir})}.$$
\end{defn}
We mimic \cite{TT2} to conjecture:
\begin{con}\label{con_JS_wall_crossing_VW}(\cite[Conjecture 5.10]{Jiang_2019})
If $H^{0,1}(S)=H^{0,2}(S)=0$, there exist rational numbers $\VW^{\tw}_{\alpha_i}(\SS)$ such that 
$$
\sP^{\perp,\tw}_{\textbf{c}}(m)=\sum_{\substack{\ell\geq 1, (\alpha_i=\delta_i \alpha)_{i=1}^{\ell}: \\
\delta_i>0,\sum_{i=1}^{\ell}\delta_i=1}} \frac{(-1)^{\ell}}{\ell !}\prod_{i=1}^{\ell}
(-1)^{\chi(\textbf{c}_i(m))}\cdot \chi(\textbf{c}_i(m))\cdot \VW^{\tw}_{\textbf{c}_i}(\SS).
$$
for $m>>0$. When either of $H^{0,1}(S)$ or $H^{0,2}(S)$ is nonzero, we only take the first term in the sum:
$$\sP^{\perp,\tw}_{\textbf{c}}(m)=
(-1)^{\chi(\textbf{c}(m))-1}\cdot \chi(\textbf{c}(m))\cdot \VW^{\tw}_{\textbf{c}}(\SS).$$
\end{con}
This conjecture is similar to Conjecture 6.5 in \cite{TT2}. 

We use the techniques for $\XX=\SS\times \cc$ and $\YY=\SS\times E$ developed in 
\cite{Jiang-Tseng_multiple} to prove this conjecture for K3 gerbes.  
First we can use the same techniques in \S 3  and \S 4 of \cite{Jiang-Tseng_multiple} to Joyce-Song pairs on $\XX$ and $\YY$ with charge class 
$\textbf{c}$ which is the pushforward from a K3 gerbe fiber $\SS$. So similar to the arguments in \cite[Page 18]{MT}, and the arguments of Behrend function in \S 4  of \cite{Jiang-Tseng_multiple}, we patch the Joyce-Song twisted pairs whose underlying sheaves have the same reduced Hilbert polynomials. The local 
isomorphisms of twisted Joyce-Song pairs on $\SS\times\cc$ and $\SS\times E$ induces local isomorphisms of 
$\sP_{\textbf{c}}^{\perp, \tw}(\XX)$ and $\sP_{\textbf{c}}^{\perp, \tw}(\YY)/E$, and the symmetric perfect obstruction theory on $\sP_{\textbf{c}}^{\perp, \tw}(\XX)$ and Oberdieck's reduced obstruction theory on 
 $\sP_{\textbf{c}}^{\perp, \tw}(\YY)/E$ is identified. 
 Then the same degeneration formula in \cite{LW}, \cite{Zhou} gives:
 
 \begin{prop}\label{prop_JS_pair_partition}
 We have:
 \begin{equation}\label{eqn_JS_pair_partition}
 \sum_{\textbf{c}}\sP^{\perp,\tw}_{\textbf{c}}(m)q^{\textbf{c}}
 =-\log\left(1+\sum_{\textbf{c}}\widetilde{\sP}^{\tw}_{\textbf{c}}(m)q^{\textbf{c}}\right)
 \end{equation}
 where the sums take over all $\textbf{c}\neq 0$ which are multiples of a fixed primitive class $\textbf{c}_0$. 
 \end{prop}
 
 \begin{thm}\label{thm_conjecture_5.10_K3_gerbe}
 Conjecture \ref{con_JS_wall_crossing_VW} holds for K3-gerbes and 
 $$\VW^{\tw}_{\textbf{c}}(\SS)=\vw^{\tw}_{\textbf{c}}(\SS).$$
 \end{thm}
 \begin{proof}
 Since 
 $$
 \sP^{\tw}_{\textbf{c}}(m)=(-1)^{h^0(K_{\SS})}\widetilde{\sP}^{\tw}_{\textbf{c}}(m)$$
 we have from Proposition \ref{prop_JS_pair_partition},
 \begin{align*}
 1-\sum_{\textbf{c}\in\nn \textbf{c}_0}\sP^{\tw}_{\textbf{c}}(m) q^{\textbf{c}}
 &=\exp\left(-\sum_{\textbf{c}\in\nn\textbf{c}_0}\sP^{\perp,\tw}_{\textbf{c}}(m)q^{\textbf{c}}\right)\\
 &=1+\sum_{\ell\geq 1}\frac{1}{\ell !}\cdot \sum_{\textbf{c}_1, \cdots, \textbf{c}_{\ell}\in\nn\textbf{c}_0}
 \prod_{i=1}^{\ell}\left(-\sP^{\perp,\tw}_{\textbf{c}_i}(m)q^{\textbf{c}_i}\right).
\end{align*}
So we have 
\begin{equation}\label{eqn_wall_crossing_Ptw}
-\sP^{\tw}_{\textbf{c}}(m)=\sum_{\substack{\ell\geq 1, (\textbf{c}_i=\delta_i\textbf{c})_{i=1}^{\ell}\\
\sum_{i=1}^{\ell}=1}}
\frac{(-1)^{\ell}}{\ell !}\prod_{i=1}^{\ell}\sP^{\perp,\tw}_{\textbf{c}_i}(m)
\end{equation}
Comparing with (\ref{eqn_Joyce-Song_wall_crossing_L}) before, we have
$$\sP^{\perp,\tw}_{\textbf{c}}(m)=-(-1)^{\chi(\textbf{c}(m))}\chi(\textbf{c}(m))\vw^{\tw}_{\textbf{c}}(\SS)$$
and $\VW^{\tw}_{\textbf{c}}(\SS)=\vw^{\tw}_{\textbf{c}}(\SS)$.
 \end{proof}

\subsection{All ranks $\SU(r)/\zz_r$-Vafa-Witten invariants for K3 surfaces}\label{subsec_higher_rank_SUrZr}

In this section we calculate the $\SU(r)/\zz_r$-Vafa-Witten invariants for a K3 surface $S$ defined in Definition \ref{defn_SU/Zr_VW_intro} for any rank $r\in\zz_{>0}$.

\subsubsection{Trivial $\mu_r$-gerbes}

Let $\SS_0\to S$ be a trivial gerbe over a K3 surface S. Then any $\SS_0$-twisted semistable sheaf or Higgs sheaf on $\SS_0$ is a sheaf on $\SS_0$, which is a pullback from a semistable sheaf or Higgs sheaf on $S$. 
Therefore from \cite[Proposition 6.8]{Jiang_2019}, 
the  moduli stack $\N^{ss,\tw}_{\SS_0}(r, \sO, c_2)$ of $\SS_0$-twisted semistable Higgs sheaves is isomorphic to the moduli stack $\N^{ss}_{S}(r,\sO,c_2)$ of Gieseker semistable Higgs sheaves on $S$.
Then the twisted Vafa-Witten invariants are the same the Vafa-Witten invariants in \cite{TT2}. 

From \cite[Theorem 5.24]{TT2}, we have:

\begin{thm}\label{thm_TT2_thm1.7}(\cite[Theorem 5.24]{TT2})
Let 
$$\Z_{r,0}(\SS_0, \SU(r); q)=\sum_{c_2}\VW^{\tw}_{r,0,c_2}(\SS_0)q^{c_2}$$
be the generating function of the twisted Vafa-Witten invariants for the trivial gerbe $\SS_0\to S$ to a K3 surface $S$. Then 
\begin{equation}\label{eqn_trivial_gerbe_partition_function}
\Z_{r,0}(\SS_0, \SU(r); q)=\sum_{d|r}\frac{d}{r^2}\sum_{j=0}^{d-1}G(e^{\frac{2\pi ij}{d}}q^{\frac{r}{d^2}})
\end{equation}
where $G(q):=\eta(q)^{-24}$ and $\eta(q)=q^{\frac{1}{24}}\prod_{k>0}(1-q^k)$ is the Dedekind eta function. 
\end{thm}

\subsubsection{Non trivial essentially trivial $\mu_r$-gerbes}

Let $\SS_{\ess}\to S$ be a non trivial essentially  trivial $\mu_r$-gerbe over a K3 surface S corresponding to a nontrivial line bundle $L\in\Pic(S)$. 
Then we list \cite[Proposition 6.9]{Jiang_2019}:
\begin{prop}\label{prop_moduli_essential_trivial_coarse_moduli_K3}(\cite[Proposition 6.9]{Jiang_2019})
Let $p:\SS_{\ess}\to S$ be a non-trivial essentially trivial $\mu_r$-gerbe corresponding to the line bundle $L\in\Pic(S)$,  then  
the moduli stack $\N^{ss,\tw}_{\SS_{\ess}}(r, L, c_2)$ of $\SS_{\ess}$-twisted Higgs sheaves on $\SS$ is isomorphic to the moduli stack 
 $\N^{ss,\tw,L^{\vee}}_{S}(r,\sO, c_2)$ of  $L^\vee$-twisted  Higgs sheaves on $S$.
\end{prop}

Then the moduli stack 
 $\N^{ss,\tw,L^{\vee}}_{S}(r,\sO, c_2)$ of  $L^\vee$-twisted  Higgs sheaves on $S$ is isomorphic to the moduli stack $\N^{ss,\tw}_{S}(r,L^{\vee}, c_2)$ of semistable Higgs sheaves. 
 The twisted Vafa-Witten invariants $\VW^{\tw}(\SS_{\ess})$ are the same as the Vafa-Witten invariants $\VW_{r,L,c_2}(S)$ of Tanaka-Thomas \cite{TT2}. 
 We let $g_{\ess}\in H^2(S,\mu_r)$ representing this line bundle 
 $L$ by the exact sequence:
 $$\cdots\to H^1(S,\sO_S^*)\to H^2(S,\mu_r)\to H^2(S,\sO_S^*)\to \cdots$$
 We write this data as:
 $$\textbf{c}=(r, g_{\ess}, c_2).$$
 Let $s:=|g_{\ess}|$ be the order of $g_{\ess}$ in $H^2(S,\mu_r)$.  Then we say 
 $k|\textbf{c}$ if and only if $k|r$, $k|s$ and $k|c_2$.  The twisted Vafa-Witten invariants 
 $$\VW^{\tw}_{\textbf{c}}(\SS_{\ess})=\VW_{\textbf{c}}(S)$$
 where $\VW_{\textbf{c}}(S)$ is the Vafa-Witten invariants in \cite{TT2}.  So this invariant $\VW_{\textbf{c}}(S)$ is just the invariant 
 $J(\textbf{c})$ on $\XX$, and by the general multiple cover formula for K3 surfaces in \cite{Toda_JDG}:
 $$\VW^{\tw}_{\textbf{c}}(\SS_{\ess})=\sum_{d|\textbf{c}, d\geq 1}\frac{1}{d^2}\chi(\Hilb^{1-\frac{1}{2}\chi(\textbf{c}/d, \textbf{c}/d)})$$
 We calculate:
 $$1-\frac{1}{2}\chi(\textbf{c}/d, \textbf{c}/d)
 =\frac{s^2}{2 d^2}(1-\frac{r}{d})+\frac{r}{d}(\frac{c_2}{d}-\frac{r}{d})+1.
 $$
 
 Now we fix $(r, g_{\ess})$ and let 
 $$d_{g_{\ess}}:=\gcd(r,s)\geq 1.$$
 We write:
 $$r=d_{g_{\ess}}\cdot \overline{r}; \quad  s=d_{g_{\ess}}\cdot \overline{s}.$$
 If we have a $d|d_{g_{\ess}}$, then we let 
 $$d_{g_{\ess}}=d\cdot \overline{d}_{g_{\ess}}; \quad  r=d\cdot \overline{d}_{g_{\ess}}\cdot \overline{r}; \quad s=d\cdot \overline{d}_{g_{\ess}}\cdot \overline{s}.$$
 We calculate:
 \begin{align*}
 \Z_{r, g_{\ess}}(\SS_{\ess}; q)&=\sum_{c_2}\VW^{\tw}_{\textbf{c}}(\SS_{\ess}) q^{c_2}\\
 &=\sum_{\substack{d| d_{g_{\ess}}\\
 d\geq 1, d\in\zz_{\geq 1}}}\frac{1}{d^2}\sum_{c_2/d}\chi\left(\Hilb^{\frac{s^{\tiny2}}{2 d^2}(1-\frac{r}{d})+\frac{r}{d}(\frac{c_2}{d}-\frac{r}{d})+1}(S)\right) q^{c_2}\\
&=
\sum_{\substack{d| d_{g_{\ess}}\\
 d\geq 1, d\in\zz_{\geq 1}}}\frac{1}{d^2}\sum_{c_2/d=m\in\zz}\chi\left(\Hilb^{\frac{s^{\tiny2}}{2 d^2}(1-\frac{r}{d})+\frac{r}{d}(m-\frac{r}{d})+1}(S)\right) q^{md}\\
 &=
 \sum_{\substack{d| d_{g_{\ess}}\\
 d\geq 1, d\in\zz_{\geq 1}}}\frac{1}{d^2}\sum_{m-r/d=n\in\zz}\chi\left(\Hilb^{\frac{s^{\tiny2}}{2 d^2}(1-\frac{r}{d})+\frac{r}{d}n+1}(S)\right) q^{nd+r}\\
 &=\sum_{\substack{d| d_{g_{\ess}}\\
 d\geq 1, d\in\zz_{\geq 1}}}\frac{1}{d^2}
 \sum_{n\in\zz}\chi\left(\Hilb^{n\cdot \overline{r}\cdot \overline{d}_{g_{\ess}}+1+\frac{1}{2}\overline{d}^2_{g_{\ess}}\overline{s}^2(1- \overline{d}_{g_{\ess}}\cdot \overline{r})}(S)\right) q^{nd+r}\\
 \end{align*}
 
 We first calculate
$$q^r\cdot  \sum_{n\in\zz}\chi\left(\Hilb^{n\cdot \overline{r}\cdot \overline{d}_{g_{\ess}}+1+\Diamond}(S)\right) q^{n\cdot \frac{d_{g_{\ess}}}{\overline{d}_{g_{\ess}}}}$$
where 
$\Diamond:=\frac{1}{2}\overline{d}^2_{g_{\ess}}\overline{s}^2(1- \overline{d}_{g_{\ess}}\cdot \overline{r})$. We have:
$$
\sum_{n\in\zz}\chi\left(\Hilb^{n\cdot \overline{r}\cdot \overline{d}_{g_{\ess}}+1+\Diamond}(S)\right) \left(q^{\frac{d_{g_{\ess}}}{\overline{r}\cdot\overline{d}^2_{g_{\ess}}}}\right)^{n\cdot\overline{r}\cdot\overline{d}_{g_{\ess}}}=
\left(q^{\frac{d_{g_{\ess}}}{\overline{r}\cdot\overline{d}^2_{g_{\ess}}}}\right)^{-\Diamond}\cdot 
\frac{1}{\overline{r}\overline{d}_{g_{\ess}}}\sum_{m=0}^{\overline{r}\overline{d}_{g_{\ess}}-1}
G\left(e^{\frac{2\pi im}{\overline{r}\overline{d}_{g_{\ess}}}}q^{\frac{d_{g_{\ess}}}{\overline{r}\cdot\overline{d}^2_{g_{\ess}}}}\right)
$$

We finally get:
\begin{thm}\label{thm_essential_trivial_VW_partition}
We have:
 $$\Z_{r, g_{\ess}}(\SS_{\ess}; q)= \sum_{\substack{d| d_{g_{\ess}}\\
 d\geq 1, d\in\zz_{\geq 1}}}
 q^{r}\cdot \frac{\overline{d}_{g_{\ess}}}{d_{g_{\ess}}}\cdot \frac{1}{r}\cdot 
 \sum_{m=0}^{\overline{r}\overline{d}_{g_{\ess}}-1}
G\left(e^{\frac{2\pi im}{\overline{r}\overline{d}_{g_{\ess}}}}q^{\frac{d_{g_{\ess}}}{\overline{r}\cdot\overline{d}^2_{g_{\ess}}}}\right).
$$
\end{thm}

\begin{rmk}\label{rmk_essentially_trivial_VW}
If $r$ is a prime number, then for any $(r, g_{\ess}, c_2)$, $d_{g_{\ess}}=1$, $\overline{d}_{g_{\ess}}=1$, and 
$\overline{r}=r$, then we have:
$$\Z_{r, g_{\ess}}(\SS_{\ess}; q)= q^{r}\cdot \frac{1}{r}\cdot 
 \sum_{m=0}^{r-1}
G\left(e^{\frac{2\pi im}{r}}q^{\frac{1}{r}}\right).
$$
\end{rmk}

\subsubsection{Non essentially trivial $\mu_r$-gerbes}

Let $\SS\to S$ be a non essentially trivial $\mu_r$-gerbe and let $g_{\opt}\in H^2(S,\mu_r)$ be its class. 
Suppose that 
$$o_{\opt}:=|\varphi(\SS_{\opt})|=|\alpha|$$
is the order in $H^2(S,\sO_S^*)_{\tor}$. Then $o_{\opt}|r$. 
The gerbe $\SS_{\opt}\to S$ has a decomposition 
$$\SS\to \SS_{\opt}\stackrel{\overline{p}}{\longrightarrow} S$$
where the first map is a trivial $\mu_{\frac{r}{o_{\opt}}}$-gerbe and the second map $\overline{p}$ is an optimal 
$\mu_{o_{\opt}}$-gerbe over $S$.  So in this case we are dealing with rank $r$ semistable Higgs twisted sheaves 
for the optimal gerbe $\overline{p}: \SS_{\opt}\to S$. 
Note that if $o_{\opt}=r$, then $p: \SS\to S$ is an optimal $\mu_r$-gerbe. 
In the next we just take  $\SS_{\opt}\to S$ as an optimal $\mu_{o_{\opt}}$-gerbe. 

From \cite[Theorem 6.13]{Jiang_2019} (which also holds for higher rank case),  the moduli stack 
 $\N^{s,\tw}_{\SS_{\opt}}(r,\sO, c_2)$ of  stable $\SS_{\opt}$-twisted  Higgs sheaves on $\SS_{\opt}$ is a $\mu_{\opt}$-gerbe over the moduli stack $\N^{s,(P, G)}_{\SS_{\opt}}(r,0, c_2)$ of stable Yoshioka Higgs sheaves. 
In the semistable case we define in \S \ref{subsec_generalized_twisted_VW}, 
$$\VW^{\tw}_{\textbf{c}}(\SS_{\opt})=\vw^{\tw}_{\textbf{c}}(\SS_{\opt})=\JS^{\tw}(\textbf{c})$$
for $\textbf{c}=(r,0,c_2)$, where $c_2\in \frac{1}{o_{\opt}}\zz$.
Since the invariant $\JS^{\tw}(\textbf{c})$ is defined by $\epsilon_{\omega,\XX}(\textbf{c})$ which is virtually indecomposable,  and this invariant is defined in the category 
$\Coh_{\pi}^{\tw}(\XX)$, it also defines the invariants  
$\JS_{(X,\alpha)}(\textbf{c})$ in the twisted category $\Coh_{\pi}^{\tw}(X,\alpha)$, with only a $\mu_{\opt}$-gerbe structure.  Also 
the invariants $\JS^{\tw}(\textbf{c})$ are defined using the Behrend function on the moduli stacks.  From \cite[Proposition 5.9]{TT2} (the proof works for twisted Higgs sheaves), the Behrend function is always $-1$.  Therefore the invariants $\JS^{\tw}(\textbf{c})$ are the same as the invariants 
$J(\textbf{c})$
and 
 by the  multiple cover formula for twisted K3 surfaces in Theorem \ref{thm_twisted_multiple_cover_intro}:
 $$\VW^{\tw}_{\textbf{c}}(\SS_{\opt})=\sum_{d|\textbf{c}, d\geq 1}\frac{1}{o_{\opt}}\cdot \frac{1}{d^2}\chi(\Hilb^{1-\frac{1}{2}\chi(\textbf{c}/d, \textbf{c}/d)}(S))$$
 We calculate:
 $$1-\frac{1}{2}\chi(\textbf{c}/d, \textbf{c}/d)
 =\frac{r}{d}(\frac{c_2}{d}-\frac{r}{d})+1.
 $$
Since the minimal rank for a $\XX$-twisted sheaf is $o_{\opt}$, in the above formula the decomposition of the semistable coherent sheaves have minimal summand rank at least $o_{\opt}$. 
We let 
$$r=\overline{r}\cdot o_{\opt}.$$
Then we consider all $d$ such that $(d\cdot o_{\opt})|r$. 
Therefore we have:
\begin{align*}
\Z_{r, \sO}(\SS_{\opt}; q)&=\sum_{c_2\in \frac{1}{o_{\opt}}\zz}\VW^{\tw}_{\textbf{c}}(\SS_{\opt})q^{c_2}\\
&=\sum_{(d\cdot o_{\opt})|r}\frac{1}{o_{\opt}}\cdot \frac{1}{d^2}\sum_{c_2}\chi(\Hilb^{1-\frac{1}{2}\chi(\textbf{c}/d, \textbf{c}/d)}(S))q^{c_2}\\
&=\sum_{(d\cdot o_{\opt})|r}\frac{1}{o_{\opt}}\cdot \frac{1}{d^2}\sum_{\frac{c_2}{d}}\chi(\Hilb^{\frac{r}{d}(\frac{c_2}{d}-\frac{r}{d})+1}(S))q^{c_2}
\end{align*}

We let 
$$r=d\cdot o_{\opt}\cdot \overline{d}.$$
Also if we let $v=(r,0,-b)$ be the Mukai vector, where  $b=\frac{k}{o_{\opt}}$, we have 
$c_2=b+r$ and $k\in d\zz$.
We calculate:
$$\frac{c_2}{d}=\frac{k\cdot \overline{d}}{r}+o_{\opt}\cdot \overline{d},$$
and 
$$\frac{r}{d}(\frac{c_2}{d}-\frac{r}{d})+1=\frac{k\cdot \overline{d}^2\cdot o_{\opt}}{r}+1.$$
So we have 
\begin{align*}
\sum_{\frac{c_2}{d}}\chi(\Hilb^{\frac{r}{d}(\frac{c_2}{d}-\frac{r}{d})+1}(S))q^{c_2}&=\sum_{\frac{k}{o_{\opt}}\in \frac{d}{o_{\opt}}\zz}
\chi(\Hilb^{\frac{k\cdot \overline{d}^2\cdot o_{\opt}}{r}+1}(S))q^{\frac{k}{o_{\opt}}+r}\\
&=\sum_{\substack{k\in d\zz\\
k=dm, m\in\zz}}
\chi(\Hilb^{\frac{m\cdot d\cdot\overline{d}^2\cdot o_{\opt}}{r}+1}(S))q^{\frac{md}{o_{\opt}}+r}\\
&=\sum_{m\in\zz}
\chi(\Hilb^{m\cdot \overline{d}+1}(S))q^{\frac{mr}{\overline{d}\cdot o_{\opt}^2}+r}\\
&=q^{r}\cdot \sum_{m\in\zz}\chi(\Hilb^{m\cdot \overline{d}+1}(S))\left(q^{\frac{r}{\overline{d}^2\cdot o^2_{\opt}}}\right)^{m\overline{d}}\\
&=q^{r}\cdot \frac{1}{\overline{d}}\sum_{j=0}^{\overline{d}-1}G\left(e^{\frac{2\pi i j}{\overline{d}}}\cdot q^{\frac{r}{\overline{d}^2\cdot o^2_{\opt}}}\right).
\end{align*}
 So we obtain:
 
\begin{thm}\label{thm_non_essential_trivial_VW_partition}
We have:
 $$\Z_{r, \sO}(\SS_{\opt}; q)= \sum_{(d\cdot o_{\opt})|r}\frac{1}{o_{\opt}}\cdot \frac{1}{d^2}
 q^{r}\cdot \frac{1}{\overline{d}}\sum_{j=0}^{\overline{d}-1}G\left(e^{\frac{2\pi i j}{\overline{d}}}\cdot q^{\frac{r}{\overline{d}^2\cdot o^2_{\opt}}}\right).
$$
\end{thm}

\begin{rmk}\label{rmk_non-essentially_trivial_VW}
If $r$ is a prime number, then for any $(r, 0, c_2)$, $o_{\opt}=r$, $d=1$ and $\overline{d}=1$, then we have:
$$\Z_{r, \sO}(\SS_{\opt}; q)= q^{r}\cdot \frac{1}{r}\cdot 
G\left(q^{\frac{1}{r}}\right).
$$
\end{rmk}

Then from the definition of $\SU(r)/\zz_r$-Vafa-Witten invariants for a K3 surface $S$, we have:

\begin{thm}\label{thm_SUrZr_K3}
Let $S$ be a smooth K3 surface. Then we have
\begin{align*}
\Z_{r,\sO}(S, \SU(r)/\zz_r; q)&=q^r\cdot\sum_{d|r}\frac{d}{r^2}\sum_{j=0}^{d-1}G(e^{\frac{2\pi ij}{d}}q^{\frac{r}{d^2}})\\
&+\sum_{\substack{g_{\ess}\in H^2(S,\mu_r)\\
g_{\ess}\text{~algebraic~}}}\sum_{\substack{d| d_{g_{\ess}}\\
 d\geq 1, d\in\zz_{\geq 1}}}
 q^{r}\cdot \frac{\overline{d}_{g_{\ess}}}{d_{g_{\ess}}}\cdot \frac{1}{r}\cdot 
 \sum_{m=0}^{\overline{r}\overline{d}_{g_{\ess}}-1}
G\left(e^{\frac{2\pi im}{\overline{r}\overline{d}_{g_{\ess}}}}q^{\frac{d_{g_{\ess}}}{\overline{r}\cdot\overline{d}^2_{g_{\ess}}}}\right)\\
&+\sum_{\substack{g_{\opt}\in H^2(S,\mu_r)\\
g_{\opt}\text{~non-algebraic~}}}
\sum_{(d\cdot o_{\opt})|r}\frac{1}{o_{\opt}}\cdot \frac{1}{d^2}
 q^{r}\cdot \frac{1}{\overline{d}}\sum_{j=0}^{\overline{d}-1}G\left(e^{\frac{2\pi i j}{\overline{d}}}\cdot q^{\frac{r}{\overline{d}^2\cdot o^2_{\opt}}}\right).
\end{align*}
\end{thm}
\begin{proof}
This is from calculations in 
Theorem \ref{thm_TT2_thm1.7},
Theorem \ref{thm_essential_trivial_VW_partition},  and 
Theorem \ref{thm_non_essential_trivial_VW_partition}.
\end{proof}

\begin{rmk}
If $r$ is a prime number, from  Remark \ref{rmk_essentially_trivial_VW}, Remark \ref{rmk_non-essentially_trivial_VW}, and Tanaka-Thomas's calculation, we have:
$$\Z_{r,\sO}(S, \SU(r)/\zz_r; q)=\frac{1}{r^2}q^r G(q^r)+q^r\left(r^{21} G(q^{\frac{1}{r}})+r^{\rho(S)-1} \left(\sum_{m=1}^{r-1}G\left(e^{\frac{2\pi i m}{r}}q^{\frac{1}{r}}\right)\right) \right)$$
since there are $r^{\rho(S)}$ algebraic classes in $H^2(S,\mu_r)$, and $r^{22}-r^{\rho(S)}$ number of non-algebraic classes, where $\rho(S)$ is the Picard number of $S$. 
This is exactly the formula (6.36) in \cite{Jiang_2019}.
\end{rmk}

\subsection{Vafa-Witten's conjecture}

In this section we provide a generalization of Vafa-Witten's prediction of K3 surfaces for all ranks. 
In \cite[\S 6.7]{Jiang_2019}, \cite{JK} the author proved the Vafa-Witten S-duality conjecture for K3 surfaces in the prime rank case. 
We use the same method to prove the S-duality conjecture of any rank for K3 surfaces. 

Recall that the twisted Vafa-Witten invariant of an optimal  $\mu_{o_{\opt}}$-gerbe (by multiple cover formula):

 $$\VW^{\tw}_{\textbf{c}}(\SS_{\opt})=\sum_{d|\textbf{c}, d\geq 1}\frac{1}{o_{\opt}}\cdot \frac{1}{d^2}\chi(\Hilb^{1-\frac{1}{2}\chi(\textbf{c}/d, \textbf{c}/d)}(S))$$
 Then 
for any $o_{\opt}$-th root of unity $e^{2\pi i \frac{s}{o_{\opt}}} (0\leq s< o_{\opt}-1)$,  we define 
\begin{align*}
&Z_{r,\sO}(\SS_{\opt}, (e^{\frac{2\pi i s}{o_{\opt}}})^{o_{\opt}}\cdot q)\\
&:=\sum_{c_2\in\frac{1}{o_{\opt}}\zz}\vw^{\tw}_{\textbf{v}}(\SS_{\opt})(e^{2\pi i \frac{s}{o_{\opt}}})^{o_{\opt}c_2}q^{c_2}\\
&=\sum_{(d\cdot o_{\opt})|r}\frac{1}{o_{\opt}}\cdot \frac{1}{d^2}\chi(\Hilb^{1-\frac{1}{2}\chi(\textbf{c}/d, \textbf{c}/d)}(S))
\left(e^{2\pi i \frac{s}{o_{\opt}}}\right)^{\frac{\overline{d}\cdot o_{\opt}}{d}(\frac{k}{o_{\opt}}+r)}q^{\frac{k}{o_{\opt}}+r}
\end{align*}
where $c_2=b+r$ and $b=\frac{k}{o_{\opt}}$ for $k\in\zz$.
Then: 
\begin{lem}\label{lem_optimal_mur_K3}
Let $\SS_{\opt}\to S$ be a $\mu_r$-gerbe such that its class 
$\varphi[\SS_{\opt}]\in H^2(S,\sO_S^*)_{\tor}$ in the cohomological Brauer group is $o_{\opt}$ and $o_{\opt}|r$.
Then we have 
$$Z_{r,\sO}(\SS_{\opt}, (e^{\frac{2\pi i s}{o_{\opt}}})^{o_{\opt}}\cdot q)=
\sum_{(d\cdot o_{\opt})|r}\frac{1}{o_{\opt}}\cdot \frac{1}{d^2}
 q^{r}\cdot \frac{1}{\overline{d}}\sum_{j=0}^{\overline{d}-1}G\left(e^{\frac{2\pi i j}{\overline{d}}}\cdot \left(e^{\frac{2\pi i s}{o_{\opt}}}\right)\cdot q^{\frac{r}{\overline{d}^2\cdot o^2_{\opt}}}\right).$$
\end{lem}
\begin{proof}
We follow the same calculation in Theorem \ref{thm_non_essential_trivial_VW_partition}. 
\begin{align*}
&\Z_{r, \sO}(\SS_{\opt}; (e^{\frac{2\pi i s}{o_{\opt}}})^{o_{\opt}}\cdot q)\\
&=\sum_{(d\cdot o_{\opt})|r}\frac{1}{o_{\opt}}\cdot \frac{1}{d^2}
\sum_{c_2}\chi(\Hilb^{1-\frac{1}{2}\chi(\textbf{c}/d, \textbf{c}/d)}(S))
\left(e^{2\pi i \frac{s}{o_{\opt}}}\right)^{\frac{\overline{d}\cdot o_{\opt}}{d}(\frac{k}{o_{\opt}}+r)}q^{\frac{k}{o_{\opt}}+r}\\
&=\sum_{(d\cdot o_{\opt})|r}\frac{1}{o_{\opt}}\cdot \frac{1}{d^2}\sum_{\frac{k}{o_{\opt}}\in \frac{d}{o_{\opt}}\zz}
\chi(\Hilb^{\frac{k\cdot \overline{d}^2\cdot o_{\opt}}{r}+1}(S))(e^{2\pi i \frac{s}{o_{\opt}}})^{\frac{\overline{d}(k+r\cdot o_{\opt})}{d}} q^{\frac{k}{o_{\opt}}+r}\\
&=\sum_{(d\cdot o_{\opt})|r}\frac{1}{o_{\opt}}\cdot \frac{1}{d^2}
\sum_{\substack{k\in d\zz\\
k=dm, m\in\zz}}
\chi(\Hilb^{\frac{m\cdot d\cdot\overline{d}^2\cdot o_{\opt}}{r}+1}(S))(e^{2\pi i \frac{s}{o_{\opt}}})^{m\overline{d}} q^{\frac{md}{o_{\opt}}+r}\\
&=\sum_{(d\cdot o_{\opt})|r}\frac{1}{o_{\opt}}\cdot \frac{1}{d^2}
\sum_{m\in\zz}
\chi(\Hilb^{m\cdot \overline{d}+1}(S))(e^{2\pi i \frac{s}{o_{\opt}}})^{m\overline{d}}  q^{\frac{mr}{\overline{d}\cdot o_{\opt}^2}+r}\\
&=\sum_{(d\cdot o_{\opt})|r}\frac{1}{o_{\opt}}\cdot \frac{1}{d^2}
q^{r}\cdot \sum_{m\in\zz}\chi(\Hilb^{m\cdot \overline{d}+1}(S))
\left((e^{2\pi i \frac{s}{o_{\opt}}})
q^{\frac{r}{\overline{d}^2\cdot o^2_{\opt}}}\right)^{m\overline{d}}\\
&=\sum_{(d\cdot o_{\opt})|r}\frac{1}{o_{\opt}}\cdot \frac{1}{d^2}
q^{r}\cdot \frac{1}{\overline{d}}\sum_{j=0}^{\overline{d}-1}G\left(e^{\frac{2\pi i j}{\overline{d}}}\cdot(e^{2\pi i \frac{s}{o_{\opt}}}) \cdot q^{\frac{r}{\overline{d}^2\cdot o^2_{\opt}}}\right).
\end{align*}
\end{proof}

From the above calculation, we only pick up  terms in the above sum in Lemma \ref{lem_optimal_mur_K3}, which are: 
$d\cdot o_{\opt}=r$ such that $d=\frac{r}{o_{\opt}}$, and $\overline{d}=1$; and 
$d\cdot o_{\opt}=e|r$ such that $\overline{d}=\frac{r}{e}$, and $\overline{d}>1$.
We all pick up the term $j=0$ in $0\leq j\leq \overline{d}-1$. 
\begin{defn}\label{defn_Zprime_opt}
For a $g\in H^2(S,\mu_r)$ such that $|g|=o_{\opt}|r$ and $o_{\opt}\neq r$, 
we define:
\begin{align*}
\Z^\prime_{r, \sO}(\SS_{\opt}; (e^{\frac{2\pi i s}{o_{\opt}}})^{o_{\opt}}\cdot q)&:=
\sum_{\substack{(d\cdot o_{\opt})|r, r=d\cdot o_{\opt}\cdot \overline{d}\\
\overline{d}\geq 1, 
d>1}}\frac{1}{o_{\opt}}\cdot \frac{1}{d^2}
q^{r}\cdot \frac{1}{\overline{d}}G\left((e^{2\pi i \frac{s}{o_{\opt}}}) \cdot q^{\frac{r}{\overline{d}^2\cdot o^2_{\opt}}}\right).
\end{align*}

For  $g\in H^2(S,\mu_r)$ such that $|g|=o_{\opt}=r$, 
we define:
\begin{align*}
\Z^\prime_{r, \sO}(\SS_{\opt}; (e^{\frac{2\pi i s}{o_{\opt}}})^{o_{\opt}}\cdot q)&:=
\frac{1}{o_{\opt}}\cdot
q^{r}\cdot G\left((e^{2\pi i \frac{s}{o_{\opt}}}) \cdot q^{\frac{r}{o^2_{\opt}}}\right).
\end{align*}
\end{defn}

Let $Z_{r,0}(q)$ denote the Tanaka-Thomas partition function for $S$ with rank $r$ and trivial determinant 
$\sO$.
\begin{thm}\label{thm_SUr/Zr_Vafa-Witten_K3}
Let $S$ be a smooth complex K3 surface. For a  positive integer $r$, we define
\begin{align*}
&Z^{\prime}_{r,\sO}(S, \SU(r)/\zz_r;q)\\
&:=Z_{r,0}(q)+\sum_{o_{\opt}|r, o_{\opt\neq 1}}\sum_{\substack{0\neq g\in H^2(S,\mu_{o_{\opt}})}}\sum_{m=0}^{o_{\opt}-1}e^{\pi i \frac{o_{\opt}-1}{o_{\opt}} m g^2}
Z^\prime_{r,\sO}(\SS_{\opt}, (e^{\frac{2\pi i m}{o_{\opt}}})^{o_{\opt}}\cdot q)
\end{align*}
We call it the partition function of $SU(r)/\zz_r$-Vafa-Witten invariants, and we have:
\begin{align*}
&\Z^\prime_{r,\sO}(S, \SU(r)/\zz_r; q)=\Big[q^r\cdot \frac{1}{r^2} G(q^{r})\Big]+ q^r\cdot r^{21} \cdot G(q^{\frac{1}{r}})+q^{r}\cdot r^{10}\cdot \sum_{m=1}^{r-1} G\left(e^{2\pi i \frac{m}{r}}\cdot q^{\frac{1}{r}}\right)
 \\ \nonumber
&+
\sum_{\substack{d|r\\d\neq 1, d\neq r}}
\Big[q^r\cdot \left(\frac{r^{21}}{d^{21}}\right)\cdot \frac{1}{d^2}\cdot G\left(q^{\frac{r}{o_{\opt}^2}}\right)\Big] \\ \nonumber
&+
\sum_{\substack{e|r\\
e\neq 1,e\neq r}} \Big[\sum_{\substack{1\leq m\leq e-1,\\
 (m,e)=d>1}}q^r\cdot  o_{\opt}^{10}\cdot\frac{1}{d^2}\cdot \frac{1}{\overline{d}}\cdot G\left(e^{2\pi i\frac{s}{o_{\opt}}} q^{\frac{d}{\overline{d}\cdot o_{\opt}}}\right)\Big] \\ \nonumber
&+
 \sum_{\substack{1\leq m\leq r-1, \\
 (m,r)=d>1}}q^r\cdot  \left(\frac{r}{d}\right)^{10}\cdot \frac{1}{d^2}\cdot G\left(e^{2\pi i\frac{s}{o_{\opt}}} q^{\frac{r}{o^2_{\opt}}}\right)
\end{align*}
This generalizes the prediction of Vafa-Witten in \cite[\S 4]{VW} to the gauge group $\SU(r)/\zz_r$ for nonprime integers $r$.  
\end{thm}

\begin{proof}
First Tanaka-Thomas partition function is:
$$Z_{r,0}(q)=q^r\cdot\sum_{e|r}\frac{e}{r^2}\sum_{j=0}^{e-1}G(e^{\frac{2\pi ij}{e}}q^{\frac{r}{e^2}}).$$

From Lemma \ref{lem_optimal_mur_K3} and Definition \ref{defn_Zprime_opt}, we calculate:
\begin{align}\label{eqn_key_sum_pre}
&Z^{\prime}_{r,\sO}(S, \SU(r)/\zz_r;q)=q^r\cdot\frac{1}{r^2}\cdot G(q^r)+ \sum_{\substack{e|r\\
e\neq 1, e\neq r}}\frac{e}{r^2}\sum_{j=0}^{e-1}G(e^{\frac{2\pi ij}{e}}q^{\frac{r}{e^2}})+q^r\cdot\frac{1}{r}\cdot\sum_{j=0}^{r-1} G(e^{\frac{2\pi ij}{r}}q^{\frac{1}{r}}) \\  \nonumber
&+\sum_{\substack{0\neq g\in H^2(S,\mu_r)\\
|g|=o_{\opt}\\
o_{\opt}|r, o_{\opt}\neq r}}\sum_{s=0}^{o_{\opt}-1}e^{\pi i \frac{o_{\opt}-1}{o_{\opt}} s g^2}
\Big[
\sum_{\substack{(d\cdot o_{\opt})|r, r=d\cdot o_{\opt}\cdot \overline{d}\\
\overline{d}\geq 1, 
d>1}}\frac{1}{o_{\opt}}\cdot \frac{1}{d^2}
q^{r}\cdot \frac{1}{\overline{d}}G\left((e^{2\pi i \frac{s}{o_{\opt}}}) \cdot q^{\frac{r}{\overline{d}^2\cdot o^2_{\opt}}}\right)\Big]\\
&+\sum_{\substack{0\neq g\in H^2(S,\mu_r)\\
|g|=r}}\sum_{s=0}^{r-1}e^{\pi i \frac{r-1}{r} s g^2}
\Big[
\frac{1}{r}\cdot 
q^{r}\cdot G\left((e^{2\pi i \frac{s}{r}}) \cdot q^{\frac{1}{r}}\right) \Big]. \nonumber
\end{align}
We rewrite it as: 
\begin{align}\label{eqn_key_sum}
&Z^{\prime}_{r,\sO}(S, \SU(r)/\zz_r;q)
=q^r\cdot\frac{1}{r^2}\cdot G(q^r)\\ \nonumber
&+q^r\cdot\frac{1}{r}\cdot\sum_{j=0}^{r-1} G(e^{\frac{2\pi ij}{r}}q^{\frac{1}{r}}) +\sum_{\substack{0\neq g\in H^2(S,\mu_r)\\
|g|=r}}\sum_{s=0}^{r-1}e^{\pi i \frac{r-1}{r} s g^2}
\Big[
\frac{1}{r}\cdot 
q^{r}\cdot G\left((e^{2\pi i \frac{s}{r}}) \cdot q^{\frac{1}{r}}\right) \Big]\\  \nonumber
&+ \sum_{\substack{e|r\\
e\neq 1, e\neq r}}\frac{e}{r^2}\sum_{j=0}^{d-1}G(e^{\frac{2\pi ij}{e}}q^{\frac{r}{e^2}})+\sum_{\substack{0\neq g\in H^2(S,\mu_r)\\
|g|=o_{\opt}\\
o_{\opt}|r, o_{\opt}\neq r}}\sum_{s=0}^{o_{\opt}-1}e^{\pi i \frac{o_{\opt}-1}{o_{\opt}} s g^2}
\Big[
\sum_{(d\cdot o_{\opt})| r,\overline{d}=1}\frac{1}{o_{\opt}}\cdot \frac{1}{d^2}
q^{r}\cdot G\left((e^{2\pi i \frac{s}{o_{\opt}}}) \cdot q^{\frac{r}{o^2_{\opt}}}\right) \\ \nonumber
&+\sum_{(d\cdot o_{\opt})|r,\overline{d}>1}\frac{1}{o_{\opt}}\cdot \frac{1}{d^2}
q^{r}\cdot \frac{1}{\overline{d}}G\left((e^{2\pi i \frac{s}{o_{\opt}}}) \cdot q^{\frac{r}{\overline{d}^2\cdot o^2_{\opt}}}\right) \Big]. \nonumber
\end{align}
where in the third line above the sum $\sum_{(d\cdot o_{\opt})| r,\overline{d}=1}$ actually contains only one term, i.e., 
the term $d$ such that $d\cdot o_{\opt}=r$. 

We need to prove that the sum in (\ref{eqn_key_sum}) is the same as 
\begin{align}\label{eqn_key_sum2}
&\Z^\prime_{r,\sO}(S, \SU(r)/\zz_r; q)=\Big[q^r\cdot \frac{1}{r^2} G(q^{r})\Big]+ 
q^{r}\cdot r^{21}\cdot G(q^{\frac{1}{r}})+q^{r}\cdot r^{10}\cdot \sum_{m=1}^{r-1} G\left(e^{2\pi i \frac{m}{r}}\cdot q^{\frac{1}{r}}\right)
 \\ \nonumber
&+
\sum_{\substack{d|r\\d\neq 1, d\neq r\\
d\cdot o_{\opt}=r}}
\Big[q^r\cdot \left(o_{\opt}^{21}\right)\cdot \frac{1}{d^2}\cdot G\left(q^{\frac{r}{o_{\opt}^2}}\right)\Big] \\ \nonumber
&+
\sum_{\substack{e|r\\
e\neq 1,e\neq r}} \Big[\sum_{\substack{1\leq m\leq e-1,\\
 (m,e)=d>1\\
d\cdot o_{\opt}=e}}q^r\cdot  o_{\opt}^{10}\cdot\frac{1}{d^2}\cdot \frac{1}{\overline{d}}\cdot G\left(e^{2\pi i\frac{s}{o_{\opt}}} q^{\frac{d}{\overline{d}\cdot o_{\opt}}}\right)\Big] \\ \nonumber
&+
 \sum_{\substack{1\leq m\leq r-1,\\
  (m,r)=d>1}}q^r\cdot  \left(\frac{r}{d}\right)^{10}\cdot \frac{1}{d^2}\cdot G\left(e^{2\pi i\frac{s}{o_{\opt}}} q^{\frac{r}{o^2_{\opt}}}\right)
\end{align}
One can easily check that  the term $q^r\cdot\frac{1}{r^2}\cdot G(q^r)$ in both (\ref{eqn_key_sum}) and (\ref{eqn_key_sum2})
matches.  

Then we look at the term involving $q^r\cdot\frac{1}{r}\cdot\sum_{m=0}^{r-1} G(e^{\frac{2\pi im}{r}}q^{\frac{1}{r}})$ in   (\ref{eqn_key_sum2}) and 
\begin{equation}\label{eqn_key_term1}
q^r\cdot\frac{1}{r}\cdot\sum_{s=0}^{r-1} G(e^{\frac{2\pi is}{r}}q^{\frac{1}{r}})+\sum_{\substack{0\neq g\in H^2(S,\mu_r)\\
|g|=r}}\sum_{s=0}^{r-1}e^{\pi i \frac{r-1}{r} s g^2}
\Big[\frac{1}{r}\cdot 
q^{r}\cdot G\left((e^{2\pi i \frac{s}{r}}) \cdot q^{\frac{1}{r}}\right)\Big]
\end{equation}
in  (\ref{eqn_key_sum}).  
It is easy to see that when $s=0$, the above sum gives  $r^{21}q^r G(q^{\frac{1}{r}})$ which matches the corresponding term in (\ref{eqn_key_sum2}).  When $s\neq 0$, we have 
$$
\sum_{\substack{g\in H^2(S,\mu_r)\\
|g|=r}}e^{\pi i \frac{r-1}{r} s g^2}=(\epsilon(s))^{22}r^{11}$$
where $\epsilon(s)=\left(\frac{s/2}{r}\right)$ if $s$ is even; and  $\epsilon(s)=\left(\frac{(s+r)/2}{r}\right)$ if $s$ is odd.  The number$\left(\frac{a}{r}\right)$ is the Legendre symbols which is 
$1$ if $a(\mod r)$ is a perfect square, and $-1$ otherwise.  Thus $(\epsilon(s))^{22}=1$. 
Therefore the term (\ref{eqn_key_term1}) gives $q^{r}\cdot r^{10}\cdot \sum_{m=1}^{r-1} G\left(e^{2\pi i \frac{m}{r}}\cdot q^{\frac{1}{r}}\right)$ which also matches the corresponding term in  (\ref{eqn_key_sum2}). 

Next we need to show that the last two lines in (\ref{eqn_key_sum}) match the last three lines in  (\ref{eqn_key_sum2}). First we show 
\begin{equation}\label{eqn_key_term2}
 \sum_{\substack{e|r\\
e\neq 1, e\neq r}}\frac{e}{r^2}\sum_{j=0}^{d-1}G(e^{\frac{2\pi ij}{e}}q^{\frac{r}{e^2}})+\sum_{\substack{0\neq g\in H^2(S,\mu_r)\\
|g|=o_{\opt}\\
o_{\opt}|r, o_{\opt}\neq r}}\sum_{s=0}^{o_{\opt}-1}e^{\pi i \frac{o_{\opt}-1}{o_{\opt}} s g^2}
\sum_{(d\cdot o_{\opt})=r}\frac{1}{o_{\opt}}\cdot \frac{1}{d^2}
q^{r}\cdot G\left((e^{2\pi i \frac{s}{o_{\opt}}}) \cdot q^{\frac{r}{o^2_{\opt}}}\right) 
\end{equation}
in  (\ref{eqn_key_sum})  is the same as 
\begin{equation}\label{eqn_key_term3}
\sum_{\substack{d|r\\d\neq 1, d\neq r\\
d\cdot o_{\opt}=r}}
\Big[q^r\cdot \left(o_{\opt}^{21}\right)\cdot \frac{1}{d^2}\cdot G\left(q^{\frac{r}{o_{\opt}^2}}\right)\Big] +
 \sum_{\substack{1\leq m\leq r-1,\\
  (m,r)=d>1}}q^r\cdot  \left(\frac{r}{d}\right)^{10}\cdot \frac{1}{d^2}\cdot G\left(e^{2\pi i\frac{s}{o_{\opt}}} q^{\frac{r}{o^2_{\opt}}}\right)
 \end{equation}
in  (\ref{eqn_key_sum2}).
This is done by replacing 
$$\sum_{\substack{e|r\\
e\neq 1, e\neq r}}\frac{e}{r^2}\sum_{j=0}^{d-1}G(e^{\frac{2\pi ij}{e}}q^{\frac{r}{e^2}})
$$
by 
$$\sum_{\substack{o_{\opt}|r\\
o_{\opt}\neq 1, o_{\opt}\neq r}}\frac{o_{\opt}}{r^2}\sum_{s=0}^{o_{\opt}-1}G(e^{\frac{2\pi is}{o_{\opt}}}q^{\frac{r}{o_{\opt}^2}}).$$
Then when $s=0$, (\ref{eqn_key_term2}) gives 
$$\sum_{\substack{d|r\\d\neq 1, d\neq r}}
\Big[q^r\cdot \left(o_{\opt}^{21}\right)\cdot \frac{1}{d^2}\cdot G\left(q^{\frac{r}{o_{\opt}^2}}\right)\Big]
$$
which is the same as the one in (\ref{eqn_key_term3}). 
When $s\neq 0$, (\ref{eqn_key_term2}) gives 
$$\sum_{\substack{o_{\opt}|r\\o_{\opt}\neq 1, o_{\opt}\neq r}}
\Big[q^r\cdot \left(o_{\opt}^{10}\right)\cdot \frac{1}{d^2}\cdot \sum_{s=1}^{o_{\opt}-1}G\left(e^{2\pi i\frac{s}{o_{\opt}}}q^{\frac{r}{o_{\opt}^2}}\right)\Big]
$$
since 
$$\sum_{\substack{g\in H^2(S,\mu_r)\\
|g|=o_{\opt}}}e^{\pi i \frac{o_{\opt}-1}{o_{\opt}} s g^2}=(\epsilon(s))^{22}o_{\opt}^{11}$$
and $(\epsilon(s))^{22}=1$, 
and this is the same as the second term in  (\ref{eqn_key_term3}).

It remains to check the last left terms in (\ref{eqn_key_sum}) and (\ref{eqn_key_sum2}).
Then we check that both two terms give
$$\sum_{\substack{e|r\\
e\neq 1,e\neq r}} \Big[\sum_{\substack{1\leq m\leq e-1,\\
 (m,e)=d>1\\
d\cdot o_{\opt}=e}}q^r\cdot  o_{\opt}^{10}\cdot\frac{1}{d^2}\cdot \frac{1}{\overline{d}}\cdot G\left(e^{2\pi i\frac{s}{o_{\opt}}} q^{\frac{d}{\overline{d}\cdot o_{\opt}}}\right)\Big], $$
also because $$\sum_{\substack{g\in H^2(S,\mu_r)\\
|g|=o_{\opt}}}e^{\pi i \frac{o_{\opt}-1}{o_{\opt}} s g^2}=(\epsilon(s))^{22}o_{\opt}^{11}$$
and $(\epsilon(s))^{22}=1$,
we are done.
\end{proof}

\subsubsection{An example}

We include an example for the rank $r=4$ in this section. The Tanaka-Thomas partition function is
$$Z_{4,\sO}(S,\SU(4);q)=q^r\cdot \frac{1}{16} G(q^4)+q^r \frac{1}{8}(G(q)+G(-q))+q^4 \frac{1}{4}\sum_{m=0}^{3}G(e^{2\pi i\frac{m}{4}} q^{\frac{1}{4}}).$$
In this case we calculate from Theorem \ref{thm_SUr/Zr_Vafa-Witten_K3} that 
\begin{align*}
&Z_{4,\sO}^\prime(S, \SU(4)/\zz_4; q)\\
&=q^4\cdot \frac{1}{16}\cdot G(q^4)+ q^4\cdot 4^{21} G(q^{\frac{1}{4}}) +q^4\cdot 4^{10}\sum_{m=1}^{3} G(e^{2\pi i\frac{m}{4}}q^{\frac{1}{4}})+
q^4\cdot 2^{21}\cdot \frac{1}{4}\cdot G(q)	
+q^4\cdot 2^{10}\cdot \frac{1}{4}\cdot G(-q).	
\end{align*}	

If from Lemma \ref{lem_optimal_mur_K3}
we define
\begin{align*}
Z_{r,\sO}(S, \SU(r)/\zz_r;q):=Z_{r,0}(q)+\sum_{o_{\opt}|r, o_{\opt\neq 1}}\sum_{\substack{0\neq g\in H^2(S,\mu_{o_{\opt}})}}\sum_{m=0}^{o_{\opt}-1}e^{\pi i \frac{o_{\opt}-1}{o_{\opt}} m g^2}
Z_{r,\sO}(\SS_{\opt}, (e^{\frac{2\pi i m}{o_{\opt}}})^{o_{\opt}}\cdot q).
\end{align*}
Then we calculate:
\begin{align*}
&Z_{4,\sO}(S, \SU(4)/\zz_4;q)=q^4\cdot \frac{1}{16} G(q^4)+ q^4\frac{1}{8}(G(q)+G(-q))
+q^4\cdot \frac{1}{4}\cdot \sum_{m=0}^{3} G(e^{2\pi i\frac{m}{4}}\cdot q^{\frac{1}{4}})\\
&+\sum_{0\neq g\in H^2(S,\mu_2)}\sum_{s=0}^{1}e^{\pi i\frac{1}{2}s g^2}\Big[
\frac{1}{4} q^4\cdot \Big[G(e^{2\pi i\frac{s}{2}}\cdot q^{\frac{1}{4}})+G(-e^{2\pi i\frac{s}{2}}\cdot q^{\frac{1}{4}})\Big]+ \frac{1}{8}\cdot q^4 G(e^{2\pi i\frac{s}{2}}\cdot q)\Big]\\
&+\sum_{0\neq g\in H^2(S,\mu_4)}\sum_{s=0}^{3}e^{\pi i\frac{3}{4}s g^2}\frac{1}{4} q^4\cdot \Big[G(e^{2\pi i\frac{s}{4}}\cdot q^{\frac{1}{4}})\Big]\\
&=q^4\cdot \frac{1}{16} G(q^4)+q^4\cdot 4^{21}\cdot G(q^{\frac{1}{4}})+
q^4\cdot 4^{10}\cdot\sum_{m=1}^{3} G(e^{2\pi i\frac{m}{4}}q^{\frac{1}{4}})\\
&+q^4\cdot \frac{1}{8}\cdot 2^{22} G(q)+q^4\cdot \frac{1}{8}\cdot 2^{11} G(-q)
+ \sum_{0\neq g\in H^2(S,\mu_2)}\sum_{s=0}^{1}e^{\pi i\frac{1}{2}s g^2}
\frac{1}{4} q^4\cdot \Big[G(e^{2\pi i\frac{s}{2}}\cdot q^{\frac{1}{4}})+G(-e^{2\pi i\frac{s}{2}}\cdot q^{\frac{1}{4}})\Big].
\end{align*}
It is seen that the partition function $Z_{4,\sO}^\prime(S, \SU(4)/\zz_4; q)$ is part of the partition function $Z_{4,\sO}(S, \SU(4)/\zz_4;q)$, and 
$Z_{4,\sO}^\prime(S, \SU(4)/\zz_4; q)$ matches the $S$-transformation formula  of Tanaka-Thomas's partition function $Z_{4,\sO}(S, \SU(4); q)$ of the Vafa-Witten invariants.

%%%%%%%%%%%%%

%%%----------------------------------------------------------------------
%%%----------------------------------------------------------------------

%\subsection*{}

% ------------------------------------------------------------------------
\end{document}